\documentclass[11pt]{amsart}
\usepackage{amsmath,amssymb,mathrsfs,bm,enumerate}
\usepackage{yhmath} 

\usepackage[pagebackref, colorlinks=true, linkcolor=black, citecolor=black, urlcolor=black]{hyperref}
\usepackage{tikz-cd} 
\usepackage[all]{xy}
\tikzset{cong/.style={draw=none,edge node={node [sloped, allow upside down, auto=false]{$\simeq$}}},
         Isom/.style={draw=none,every to/.append style={edge node={node [sloped, allow upside down, auto=false]{$\simeq$}}}}}   

\usepackage{stmaryrd} 

\usepackage[margin=1.0in]{geometry}

\newcommand{\excise}[1]{}

\newtheorem{thm}{Theorem}

\newtheorem{theorem}{Theorem}[section]
\newtheorem{lemma}[theorem]{Lemma}
\newtheorem{proposition}[theorem]{Proposition}
\newtheorem{corollary}[theorem]{Corollary}

\newtheorem*{thm*}{Theorem}
\newtheorem*{cor*}{Corollary}

\theoremstyle{definition}
\newtheorem{definition}{Definition}

\newtheorem{remark}[thm]{Remark}

\makeatletter
\@namedef{subjclassname@2010}{%
\textup{MSC2010}}
\makeatother

\begin{document}

\title[Conformal blocks from vertex algebras and their connections on $\overline{\mathcal{M}}_{g,n}$]{Conformal blocks from vertex algebras \\ 
and their connections on $\overline{\mathcal{M}}_{g,n}$}

\author[C.~Damiolini]{Chiara Damiolini}
\address{Chiara Damiolini \newline \indent Department of Mathematics, Princeton University, Princeton, NJ 08544}
\email{chiarad@princeton.edu}

\author[A.~Gibney]{Angela Gibney}
\address{Angela Gibney \newline \indent  Department of Mathematics, Rutgers University, Piscataway, NJ 08854}
\email{angela.gibney@rutgers.edu}

\author[N.~Tarasca]{Nicola Tarasca}
\address{Nicola Tarasca 
\newline \indent Department of Mathematics \& Applied Mathematics,
\newline \indent Virginia Commonwealth University, Richmond, VA 23284 
\newline \indent {\textit{Previous address:}} Department of Mathematics, Rutgers University 
}
\email{tarascan@vcu.edu}

\subjclass[2010]{14H10, 17B69, 14C17 (primary), 81R10, 81T40, 16D90 (secondary)}
\keywords{Vertex algebras, conformal blocks and coinvariants, connections and Atiyah algebras, \linebreak \indent sheaves on moduli of curves, Chern classes of vector bundles on moduli of curves}

\begin{abstract}
We show that coinvariants of modules over  vertex operator algebras give rise to quasi-coherent sheaves on moduli  of stable pointed curves.
These  generalize  Verlinde bundles or vector bundles of conformal blocks defined using affine Lie algebras studied first by Tsuchiya--Kanie,  Tsuchiya--Ueno--Yamada, and extend work of a number of researchers.
The sheaves  carry a twisted logarithmic $\mathcal{D}$-module structure, and hence support a projectively flat connection. We identify the  logarithmic Atiyah algebra acting on them, generalizing work of Tsuchimoto for affine Lie algebras.
\end{abstract}


\maketitle



Vertex algebras
have been described as analogues of commutative associative algebras and  complex Lie algebras.  They extend constructions in the representation theory of affine Lie  algebras,  conformal field theory, finite group theory and combinatorics, integrable systems, and modular functions  \cite{m, B, Monster, DongLiMasonTwisted, bd, bzf, GS,  MModular, HuangVerlinde2005, Designs}.  

We study geometric realizations of  representations of a vertex operator algebra $V$. The idea, originating in \cite{TK, tuy}, formulated here from the perspective of \cite{bd, bzf}, is to assign  a $V$-module $M^i$ to each marked point $P_i$ on an algebraic curve $C$,  then quotient $\otimes_i M^i$ by the action of a  Lie algebra (\S \ref{PFC})  depending on $(C,P_{\bullet})$ and $V$, and thus obtain  \textit{vector spaces of coinvariants}.  
We carry this out relatively over the moduli space $\overline{\mathcal{M}}_{g,n}$ of stable pointed curves  (Def.~\ref{SOCM}). This generalizes prior work on smooth curves \cite{bzf, bd},  coinvariants obtained from affine Lie algebras \cite{TK, tuy, ts}, and special cases using vertex algebras \cite{u, NT, an1}.   
Our main result is:

\begin{thm*} \label{thmIntro}	
Given $n$ modules $M^1, \dots, M^n$ over a vertex operator algebra $V$:
\begin{enumerate}

\item The spaces of coinvariants give rise to a quasi-coherent sheaf  ${\mathbb{V}}(V; M^\bullet)$ on $\overline{\mathcal{M}}_{g,n}$.\smallskip

\item Assume moreover that all the modules $M^i$ are simple. Denote by $a_i$ the conformal dimension of $M^i$ and by $c$ be the central charge of $V$. Then, the  Atiyah algebra $\frac{c}{2}\mathcal{A}_\Lambda +\sum_{i=1}^n a_i \mathcal{A}_{\Psi_i}$ with logarithmic singularities along the divisor of singular curves in $\overline{\mathcal{M}}_{g,n}$ acts on  ${\mathbb{V}}(V; M^\bullet)$ specifying a twisted logarithmic $\mathcal{D}$-module structure. 
\end{enumerate}
\end{thm*}

The modules we use are specified in \S \ref{VMod}, Atiyah algebras are reviewed in \S \ref{LogAtiyahAlgebra}, and the theorem is proved in \S \ref{ProjConnVM}, where a more precise statement is given. As an application, when ${\mathbb{V}}(V;M^\bullet)$ is of finite rank on ${\mathcal{M}}_{g,n}$ and $c, a_i\in \mathbb{Q}$ (verified in some cases, see \S\ref{Remark} and also \cite{DGT2}), the action of the Atiyah algebra from the Theorem above 
gives the Chern character of ${\mathbb{V}}(V;M^\bullet)$ on $\mathcal{M}_{g,n}$:

\begin{cor*}
\label{chMgn}
When  $\mathbb{V}(V; M^\bullet)$ is of finite rank on $\mathcal{M}_{g,n}$,  $M^i$ are simple, and $c, a_i\in \mathbb{Q}$,
\[
{\rm ch}\left(  \mathbb{V}(V; M^\bullet)  \right) = {\rm rank} \,\mathbb{V}(V; M^\bullet) \, \cdot \, \exp\left( \frac{c}{2}\,\lambda + \sum_{i=1}^n a_i \psi_i \right)\, \in H^*(\mathcal{M}_{g,n}, \mathbb{Q}).
\]
In particular, Chern classes of $\mathbb{V}(V; M^\bullet)$ lie in the tautological ring.
\end{cor*}

The analogue of the Theorem for sheaves of coinvariants of integrable representations of an affine Lie algebra was proved in \cite{ts}.
The formula in the Corollary, 
proved in \S \ref{Chern}, extends the computation of the Chern classes of vector bundles of coinvariants on $\mathcal{M}_{g,n}$ in the affine Lie algebra case from~\cite{mop}.

\subsection{Overview}
In this work we extend the construction of sheaves of coinvariants ${\mathbb{V}}(V; M^\bullet)$ of modules over vertex operator algebras on the locus of smooth curves ${\mathcal{M}}_{g,n}$ from \cite{bzf}, to stable pointed curves  $\overline{\mathcal{M}}_{g,n}$ (see \S\ref{SecondDescent}).   Such sheaves on $\overline{\mathcal{M}}_{g,n}$ have been studied before in particular cases, the most well-known being the case of affine Lie  algebras, and also in more generality (see \S \ref{priorExtensions} for an account).  We note that \textit{conformal blocks} are defined as dual to coinvariants.  For vertex algebras arising from affine Lie  algebras, conformal blocks are known to be vector spaces canonically isomorphic to generalized theta functions \cite{bl1, Faltings, KNR, P, bl3, HuangVerlinde2005}.  

We work with  modules over \textit{vertex operator algebras}, which admit an action of the Virasoro algebra. 
We define these and related objects in \S \ref{Background}.  Vertex algebras and their modules depend on a formal variable $z$. The geometric realization from \cite{bzf} starts by considering $z$ as the formal coordinate at a point on an algebraic curve. One is thus led to consider the moduli space $\widetriangle{\mathcal{M}}_{g,n}$  parametrizing objects  $(C, P_\bullet=(P_1,\ldots,P_n), t_\bullet=(t_1,\ldots,t_n))$,  where $(C,P_\bullet)$ is a stable $n$-pointed curve of genus $g$, and $t_i$ is a formal coordinate at $P_i$ (\S\ref{AutBundles}).

The strategy is to first define sheaves of coinvariants $\widetriangle{\mathbb{V}}(V; M^\bullet)$ on $\widetriangle{\mathcal{M}}_{g,n}$ (Definition~\ref{Vtriangle}).
Then one shows that  $\widetriangle{\mathbb{V}}(V; M^\bullet)$  descends to a sheaf ${\mathbb{V}^J}(V; M^\bullet)$ on ${J}=\overline{\mathcal{J}}_{g,n}^{1,\times}$, the moduli space parametrizing points  $(C,P_\bullet, \tau_\bullet)$, where $\tau_i$ is a non-zero $1$-jet at $P_i$ (Definition~\ref{SheafOfCoinvariants}).  A second descent  allows one to define the  quasi-coherent sheaf ${\mathbb{V}}(V; M^\bullet)$ on $\overline{\mathcal{M}}_{g,n}$ (Theorem/Definition~\ref{SOCM}). Fibers of the sheaves are canonically isomorphic to the vector spaces of coinvariants   \eqref{spacecon} and \eqref{spaceconCstar}. Figure \ref{fig:Sheavesofcoinvariants}  depicts the relationships between these sheaves and spaces. 

The action of the Virasoro algebra on   $V$-modules is responsible for the presence of the twisted logarithmic $\mathcal{D}$-module structure on sheaves of coinvariants. The twisted $\mathcal{D}$-module structure of sheaves of coinvariants on families of smooth curves has been presented in \cite{bzf} as an integral part of their construction. 
Twisted logarithmic $\mathcal{D}$-module structures on sheaves over a smooth scheme $S$ are parametrized by elements in the $\mathbb{C}$-vector space $\textrm{Ext}^1(\mathcal{T}_S(-\log \Delta), \mathcal{O}_S)$, that is, the space of logarithmic Atiyah algebras \cite{besh}.
Part \textit{(ii)} of the above Theorem 
has the merit of identifying  the logarithmic Atiyah algebra acting on sheaves of coinvariants in case the modules are simple. 
For this, we use that for a simple module $M$, one has $L_0(v)=(\deg(v)+a)v$, for homogeneous $v\in M$, where $a\in \mathbb{C}$ defines the conformal dimension of $M$ (also called conformal weight).

Crucial to the identification of the logarithmic Atiyah algebra is the \textit{Virasoro uniformization}, Theorem \ref{Virunif}, which 
gives a Lie-theoretic realization of the logarithmic tangent bundle on families of stable curves (after \cite{adkp}, \cite{besh}, \cite{kvir}, \cite{tuy}).
This was proved in \cite{tuy};  we give an alternative proof of this result in \S\ref{AutBundles}, extending to families of stable curves with singularities the argument for families of smooth curves given in \cite{bzf}.

We work over the smooth DM  stack of stable curves $\overline{\mathcal{M}}_{g,n}$, in particular 
assuming  $2g-2+n>0$.

\subsection{Future directions}
\label{sec:future}
The results of this paper  serve as a cornerstone for the following developments: for a vertex operator algebra $V$ which is rational, $C_2$-cofinite, and with $V_0\cong\mathbb{C}$, in \cite{DGT2} we prove that coinvariants at singular curves satisfy the factorization property, and that the sheaf of coinvariants forms a vector bundle. Following work for affine Lie algebras in \cite{moppz}, in \cite{DGT3}, under the further assumption that V is simple and self-contragredient, we derive the Chern classes for $\mathbb{V}(V; M^\bullet)$ on $\overline{\mathcal{M}}_{g,n}$ from the above Corollary, the factorization \cite{DGT2}, 
and the  recursion solved in~\cite{moppz} (see \S \ref{sec:add}).  

Classes constructed from affine Lie algebras are known to be semi-ample in genus zero  \cite{fakhr}; shown to determine full-dimensional subcones of nef cones in all codimension, and used to produce new birational models of moduli of curves  (e.g., \cite{agss}, \cite{gg}, \cite{gjms}, \cite{ags}). We note that while the extension to the boundary of the sheaves we consider here could have been done with the technology in place when \cite{bzf} was written, the motivation to do so comes from these recent results.  As the representation theory of vertex operator algebras is richer than for affine Lie algebras, the sheaves $\mathbb{V}(V; M^\bullet)$ are expected to provide new information about  the birational geometry of $\overline{\mathcal{M}}_{g,n}$  \cite{DG4}.

\subsubsection*{Acknowledgements} The authors thank James Lepowsky and Yi-Zhi Huang for graciously and endlessly answering questions about vertex algebras, and David Ben-Zvi for a helpful discussion and email correspondence.  We thank Bin Gui for pointing out a mischaracterization of 
Heisenberg algebras in an earlier version, and Andr\'e Henriques for valuable comments.
We are indebted to the treatment of coinvariants of vertex algebras in \cite{bzf}, and of coinvariants of affine Lie algebras in \cite{bk}.
We thank the referee for a careful reading of the paper and for pointing out some corrections.
AG was supported by  NSF DMS-1201268.


\section{Background}\label{Background} 
Following \cite{bzf, Monster, fhl, lepli},  in \S \ref{Vira} we review the Virasoro algebra, in \S \ref{V} we define a vertex  operator algebra $V$, and in \S \ref{VMod} we define the $V$-modules we will work with. 
  Vertex algebras and their modules depend on a formal variable $z$;  in \S\ref{Vmodbdles} we will consider geometric realizations of vertex algebras and their modules 
independent of the variable $z$. To show we one can do this, we employ certain Lie groups and Lie algebras related to automorphisms and derivations of $\mathbb{C}\llbracket z \rrbracket$ and $\mathbb{C}(\!( z )\!)$ reviewed in \S \ref{functors}.   The Lie algebras arise as Lie sub- or quotient algebras of the Virasoro algebra. In \S \ref{UV}--\S \ref{HCcompM} we define the  Lie group $\textrm{Aut}_+\mathcal{O}$ and Lie algebra $\mathfrak{L}(V)$, emphasizing certain properties of their actions on $V$-modules.

\subsubsection*{Setting.} We work over the field of complex numbers $\mathbb{C}$ in the algebraic category. We  often reduce  locally  to discs obtained as spectra of completed local rings.

\subsection{The Virasoro algebra}
\label{Vira} 
Throughout we work with a $\mathbb{C}$-algebra $R$. For $\mathcal{K}$ the functor which assigns to $R$ the field of Laurent series $R(\!(z)\!)$,  consider the Lie algebra $\textrm{Der}\,\mathcal{K}(R)=R(\!( z)\!) \partial_z$. This is the \textit{Witt (Lie) algebra} with coefficients in $R$, and  is a complete topological Lie algebra with respect to the topology for which a basis of the open neighborhoods of $0$ is given by the subspaces $z^N R\llbracket z\rrbracket$, for $N\in \mathbb{Z}$.
The Witt algebra $\textrm{Der}\,\mathcal{K}(R)$ is topologically 
generated over $R$ by the derivations $L_p:=-z^{p+1}\partial_z$, for $p\in \mathbb{Z}$, with  relations $[L_p, L_q] = (p-q) L_{p+q}$. Let $\mathfrak{gl}_1$ be the functor which assigns to $R$ the Lie algebra $R$  with the trivial Lie bracket.  
The \textit{Virasoro (Lie) algebra} $\textrm{Vir}$ is the functor of Lie algebras defined as the central extension
\[
0\rightarrow \mathfrak{gl}_1 \cdot K \rightarrow \textrm{Vir} \rightarrow \textrm{Der}\, \mathcal{K} \rightarrow 0
\]
with bracket
\[
[K, L_p]=0, \qquad [L_p, L_q] = (p-q) L_{p+q} + \frac{K}{12} (p^3-p)\delta_{p+q,0}.
\]
Here $K$ is a formal vector generating  the center of  the Virasoro Lie algebra.

\subsection{Vertex operator algebras}
\label{V}
A \textit{vertex operator algebra} is the datum $(V, |0\rangle, \omega, Y(\cdot,z))$, where:
\begin{enumerate}
\item $V=\oplus_{i\geq 0} V_i$ is a $\mathbb{Z}_{\geq 0}$-graded  $\mathbb{C}$--vector space with $\dim V_i<\infty$;

\item $|0\rangle$ is an element in $V_0$, called the \textit{vacuum vector};

\item $\omega$ is an element in $V_2$, called the \textit{conformal vector};

\item $Y(\cdot,z)$ is a linear operation
\begin{eqnarray*}
Y(\cdot,z)\colon V &\rightarrow & \textrm{End}(V)\llbracket z,z^{-1} \rrbracket\\
A &\mapsto & Y(A,z) :=\sum_{i\in\mathbb{Z}} A_{(i)}z^{-i-1}.
\end{eqnarray*}
The series $Y(A,z)$ is called the \textit{vertex operator} assigned to $A$.
\end{enumerate}

The datum $(V, |0\rangle, \omega, Y(\cdot,z))$ is required to satisfy the following axioms:

\begin{enumerate}
\item[(a)] \textit{(vertex operators are fields)} for all $A,B\in V$, one has $A_{(i)}B=0$, for $i>\!>0$;

\item[(b)] \textit{(vertex operators and the vacuum)} one has $Y(|0\rangle, z)=\textrm{id}_V$, that is:
\[
|0\rangle _{(-1)} = \textrm{id}_V \qquad \mbox{and} \qquad |0\rangle _{(i)} = 0, \quad\mbox{for $i\neq -1$,}
\]
and for all $A\in V$, one has $Y(A,z)|0\rangle \in A+ zV\llbracket z \rrbracket$, that is:
\[
A_{(-1)}|0\rangle = A \qquad \mbox{and} \qquad A_{(i)}|0\rangle =0, \qquad \mbox{for $i\geq 0$};
\]

\item[(c)] \textit{(weak commutativity)} for all $A,B\in V$, there exists $N\in\mathbb{Z}_{\geq 0}$ such that
\[
(z-w)^N \, [Y(A,z), Y(B,w)]=0 \quad \mbox{in }\textrm{End}(V)\llbracket z^{\pm 1}, w^{\pm 1} \rrbracket;
\]

\item[(d)] \textit{(conformal structure)} the Fourier coefficients of the vertex operator $Y(\omega,z)$ satisfy the Virasoro relations:
\[
\left[\omega_{(p+1)}, \omega_{(q+1)} \right] = (p-q) \,\omega_{(p+q+1)} + \frac{c}{12} \,\delta_{p+q,0} \,(p^3-p) \,\textrm{id}_V, 
\]
for some constant $c\in\mathbb{C}$, called the \textit{central charge} of the vertex algebra. 
We can then identify each $\omega_{(p)}\in \textrm{End}(V)$ as an action of $L_{p-1}\in \textrm{Vir}$ on $V$. Moreover, one has
\[
L_0|_{V_i}=i\cdot\textrm{id}_V, \,\,\forall i \qquad \mbox{and} \qquad Y\left(L_{-1}A,z \right) = \partial_z Y(A,z).
\]
\end{enumerate}

We will often abbreviate $(V, |0\rangle, \omega,Y(\cdot,z))$ with $V$. 
The vertex operator of the conformal vector gives a representation of the Virasoro algebra on $V$, with the central element $K\in {\rm Vir}$ acting on $V$ as multiplication by the central charge $c$.  
The action of $L_0$ coincides with the \textit{degree} operator on $V$; the action of $L_{-1}$ --- called  \textit{translation}   --- is determined by $L_{-1}A=A_{(-2)}|0\rangle$, for $A\in V$.

As a consequence of the axioms, one has $A_{(i)}V_m\subseteq V_{m+\deg A -i-1}$, for homogeneous $A\in V$ (see e.g.,~\cite{zhu}). We will then say that the degree of the operator $A_{(i)}$ is 
\begin{equation}
\label{degAi}
\deg A_{(i)}:= \deg A -i-1, \qquad \mbox{for homogeneous $A$ in $V$.}
\end{equation}

Throughout, we assume that $V_0\cong \mathbb{C}$. In the literature, this  condition is often referred to as $V$ being of CFT-type. 
This condition is used in the proof of Theorem \ref{thm:POV}.

\subsection{Modules over vertex operator algebras}
\label{VMod} 
Let $(V, |0\rangle, \omega,Y(\cdot,z))$ be a vertex operator algebra. We abbreviate this $4$-tuple by  $V$.  
We describe here the  $V$-modules which will be used throughout. These are admissible $V$-modules satisfying additional properties.

A \textit{weak $V$-module} is a pair $\left(M,Y^M\right)$, consisting of  a $\mathbb{C}$-vector space  $M$ and a linear map $Y^M(\cdot,z)\colon V \rightarrow  \textrm{End}(M)\left\llbracket z,z^{-1} \right\rrbracket$ assigning to every element $A \in V$ an  $\textrm{End}(M)\llbracket z \rrbracket$-valued vertex operator $Y^M(A,z) :=\sum_{i\in\mathbb{Z}} A^M_{(i)}z^{-i-1}$. The pair $\left(M,Y^M(\cdot,z)\right)$ is required to satisfy the following axioms:
\begin{enumerate}
\item[(a)]$A^M_{(i)}v=0$ for $i>\!>0$, where $A\in V$, and $v\in M$;
\item[(b)]$Y^M(|0\rangle,z)=\rm{id}_M$;
  \item[(c)] \textit{(weak commutativity)} for all $A, B \in V$ there exists $N\in\mathbb{Z}_{\geq 0}$ such that for all $v \in M$ one has
\[
(z-w)^N \left[ Y^M(A,z), Y^M(B,w) \right]v=0;
\]

    \item[(d)] \textit{(weak associativity)} for all $A \in V$, $v \in M$, 
     there exists $N\in\mathbb{Z}_{\geq 0}$ (depending only on $A$ and~$v$) such that for all $B \in V$ one has
\[
(w+z)^N \left( Y^M(Y(A,w)B,z)- Y^M(A,w+z)Y^M(B,z)\right)v =0;
\]
    \item[(e)] \textit{(conformal structure)} the Fourier coefficients of $Y^M(\omega,z) = \sum_{i \in \mathbb{Z}} \omega_{(i)}^M z^{-i-1}$ satisfy the Virasoro relation
    \[
    \left[\omega^M_{(p+1)}, \omega^M_{(q+1)}\right] = (p-q) \,\omega^M_{(p+q+1)} + \frac{c}{12} \,\delta_{p+q,0} \,(p^3-p) \,\textrm{id}_M, 
    \] 
	where $c \in \mathbb{C}$ is the central charge of $V$.
\end{enumerate}

An \textit{admissible $V$-module}  $\left(M,Y^M\right)$ is a weak $V$-module additionally satisfying:
\begin{enumerate}
\item[(f)] $M=\oplus_{i\geq 0}M_i$ is $\mathbb{Z}_{\ge 0}$-graded, and for $A\in V$ homogeneous  and $i \in \mathbb{Z}$, one has
\begin{equation} \label{ModShift}
    A^M_{(i)} M_k \subseteq M_{k + \deg (A) -i-1}.
\end{equation}
\end{enumerate}

In this paper, we  work with admissible $V$-modules $M$ which satisfy the following conditions:
\begin{enumerate}
\item[(g)] $\dim M_i < \infty$, for all $i$;
\item[(h)] $L_0$ acts semisimply on $M$.
\end{enumerate}
For example, assumptions (g) and (h) are satisfied by finitely generated admissible modules over a \textit{rational} vertex operator algebra: indeed, when $V$ is rational, finitely generated admissible $V$-modules decompose as finite sums of simple $V$-modules, and simple $V$-modules satisfy (h) --- since they are \textit{ordinary} modules --- and (g)  (see \cite{DongLiMasonTwisted}, \cite[Rmk 2.4]{DongLiMasonRegular}). 

We note that it has been shown that weak associativity and  weak commutativity are equivalent to the \textit{Jacobi identity} (see for instance \cite{dl}, \cite{fhl},  \cite{lepli}, \cite{lvertexsuperalg}). 

Observe that condition (e) above implies that the Virasoro algebra acts on  $M$ by identifying $\omega^M_{(p+1)} \in \textrm{End}(M)$ with $L_p$ and  $c\cdot \textrm{id}\in \textrm{End}(M)$ with $K\in \operatorname{Vir}$.  

After \cite[Thm 3.5.4]{lepli} or \cite[\S 3.2.1]{bzf}, $V$ satisfies  weak associativity. In particular, $V$ is a $V$-module.

\subsubsection{}
\label{cd}
A  $V$-module $M$ is \textit{simple} if the only $V$-submodules of $M$ are itself and the $0$ module. 
The \textit{conformal dimension}  of a simple $V$-module $M$ is defined as the value $a\in \mathbb{C}$ such that
$L_0 v=(a+\deg v)v$, for homogeneous $v\in M$.

\subsection{Lie groups and Lie algebras} 
\label{functors}
Here we define a number of Lie groups and their associated Lie algebras related to automorphisms and derivations of $R\llbracket z\rrbracket$ and $R(\!( z)\!)$ for a $\mathbb{C}$-algebra $R$ and a formal variable $z$.
The topology of $R(\!( z)\!)$ is defined by taking as the basis of open neighborhoods of $0$ the subspaces $z^N R\llbracket z\rrbracket$, for $N \in \mathbb{Z}$. When we restrict to $N \geq 0$, this defines the topology of $R\llbracket z\rrbracket$.

To begin with, we consider the  group functor  represented by a group ind-scheme denoted $\underline{\textrm{Aut}}\,\mathcal{O}$ \cite[\S 6.2.3]{bzf}: this functor assigns to  $R$ the  group of continuous automorphisms of $R\llbracket z\rrbracket$:
\[
R\mapsto\left\{z\mapsto\rho(z)=a_0+a_1z+a_2z^2+\cdots \Big| \begin{array}{l}  a_i \in R, \, a_1 \mbox{ a unit},\\ a_0 \mbox{ nilpotent}
  \end{array}\right\}.
\]
 The group $\underline{\textrm{Aut}}\,\mathcal{O}(R)$ parametrizes topological generators of $R \llbracket z \rrbracket$, that is, elements $t\in R \llbracket z \rrbracket$ such that $R \llbracket z \rrbracket \cong R \llbracket t \rrbracket$.

Similarly, we consider the functor which assigns to $R$ the group of continuous automorphisms of $R\llbracket z\rrbracket$ preserving the ideal $zR\llbracket z\rrbracket$:
\[
R\mapsto\left\{z\mapsto\rho(z)=a_1z+a_2z^2+\cdots \Big| \begin{array}{l}  a_i \in R, \\  a_1 \mbox{ a unit}
  \end{array}\right\}.
\]
This functor is represented by a group scheme denoted ${\textrm{Aut}}\,\mathcal{O}$. The group ${\textrm{Aut}}\,\mathcal{O}(R)$ parametrizes topological generators of $z R \llbracket z \rrbracket$, that is, elements $t \in R \llbracket z \rrbracket$ such that $z R \llbracket z \rrbracket \cong t R \llbracket t \rrbracket$.   

Let $\textrm{Aut}_+\mathcal{O}$ be the group scheme whose Lie group of $R$-points is 
\[
\textrm{Aut}_+\mathcal{O}(R)=\left\{z\mapsto\rho(z)=z+a_2z^2+\cdots \,|\, a_i \in R \right\}.
\]

Finally, consider  the functor which assigns to $R$ the group of continuous automorphisms of~$R(\!( z)\!)$:
\[
R\mapsto\left\{z\mapsto\rho(z)=\sum_{i\geq i_0} a_i z^i \,\, \Bigg| \begin{array}{l}  a_i \in R, \, a_1 \mbox{ a unit}, \\ a_i \mbox{ nilpotent for $i\leq 0$}
  \end{array}\right\}.
\]
This group functor is represented by a group ind-scheme $\textrm{Aut}\,\mathcal{K}$ \cite[\S 17.3.4]{bzf}.

One has the following inclusions:
\[\textrm{Aut}_+\mathcal{O}(R) \subseteq \textrm{Aut}\,\mathcal{O}(R) \subseteq \underline{\textrm{Aut}}\,\mathcal{O}(R) \subseteq \textrm{Aut}\,\mathcal{K}(R).
\]

We will also consider the associated Lie algebras, obtained as the tangent space at the identity:
\begin{align*}
\textrm{Der}_+ \mathcal{O} &:= \textrm{Lie}(\textrm{Aut}_+\mathcal{O}), &
\textrm{Der}_0\, \mathcal{O} &:= \textrm{Lie}(\textrm{Aut}\,\mathcal{O}), \\
\textrm{Der}\, \mathcal{O} &:= \textrm{Lie}(\underline{\textrm{Aut}}\,\mathcal{O}), &
\textrm{Der}\, \mathcal{K} &:= \textrm{Lie}(\textrm{Aut}\,\mathcal{K}),
\end{align*}
whose $R$-points are given by:
\begin{align*}
\textrm{Der}_+ \mathcal{O}(R) &= z^2R\llbracket z\rrbracket \partial_z, &
\textrm{Der}_0\,\mathcal{O}(R) & =  zR\llbracket z\rrbracket \partial_z,\\
\textrm{Der}\,\mathcal{O}(R) & =R\llbracket z\rrbracket \partial_z, &
\textrm{Der}\,\mathcal{K}(R) & =R(\!( z)\!) \partial_z.
\end{align*}

The $2$-cocycle defining the bracket of the Virasoro algebra  $c(p,q):= \frac{1}{12} (p^3-p)\delta_{p+q,0}$ vanishes for $p,q\geq -1$, hence
the Lie algebras 
$\textrm{Der}\, \mathcal{O}(R)$, 
$\textrm{Der}_0\, \mathcal{O}(R)$, 
and $\textrm{Der}_+ \mathcal{O}(R)$ 
are Lie subalgebras of the Virasoro algebra $\textrm{Vir}(R)$.

\subsubsection{}
Observe that although one has the equalities $$\textrm{Aut}\,\mathcal{O}(\mathbb{C})=\underline{\textrm{Aut}}\,\mathcal{O}(\mathbb{C})=\textrm{Aut}\,\mathcal{K}(\mathbb{C}),$$  the $\mathbb{C}$-points of the associated Lie algebras have strict inclusions 
$$z\mathbb{C}\llbracket z\rrbracket \partial_z \subset \mathbb{C}\llbracket z\rrbracket \partial_z \subset \mathbb{C}(\!( z)\!) \partial_z.$$
For instance, the tangent vector $\partial_z$ in $\textrm{Der}\,\mathcal{O}(\mathbb{C})$ is the differential of an automorphism which is not detected over $\mathbb{C}$.

\subsection{The action of $\textup{\rm Aut}_+\mathcal{O}$ on $V$-modules}
\label{Action}
Let $V$ be a vertex operator algebra, and let $M$ be a  $V$-module.
The action of $\textrm{Vir}$ induces actions of its Lie subalgebras $\textrm{Der}\, \mathcal{O}$, $\textrm{Der}_0\, \mathcal{O}$, and $\textrm{Der}_+ \mathcal{O}$  on~$M$.

The Lie algebra $\textrm{Der}_+ \mathcal{O}$ is generated by $L_p$ with $p> 0$. After \eqref{degAi}, each operator $L_p$ with $p> 0$ has degree $-p< 0$.
Since the gradation on $M$ is bounded from below, that is, $M_i=0$ for $i<\!<0$, the action of $\exp (L_p)$ is a finite sum, for $p>0$, hence 
the action of $\textrm{Der}_+ \mathcal{O}=\textrm{Lie}(\textrm{Aut}_+\mathcal{O})$ can be exponentiated to a left action of $\textrm{Aut}_+\mathcal{O}$ on $M$. 
Moreover, each $M_{\leq i} := \oplus_{m\leq i} \, M_m$ is a finite-dimensional $(\textrm{Aut}_+\mathcal{O})$-submodule of $M$.
The representation of $\textrm{Aut}_+\mathcal{O}$ on $M$ is the inductive limit of the representations $M_{\leq i}$.

\subsection{The Lie algebra $\mathfrak{L}(V)$ associated to a vertex algebra $V$}
\label{UV}
Given a vertex algebra $V$, define $\mathfrak{L}(V)$ as the quotient
\begin{align}
\label{LV}
\mathfrak{L}(V) := \big( V\otimes \mathbb{C}(\!(t)\!) \big) \big/ \textrm{Im}\, \partial
\end{align}
where $\partial:= L_{-1}\otimes \textrm{Id}_{\mathbb{C}(\!(t)\!)} +  \textrm{Id}_V \otimes \partial_t$.
Denote by $A_{[i]}$ the projection in $\mathfrak{L}(V)$ of $A\otimes t^i\in V\otimes \mathbb{C}(\!(t)\!)$.
The quotient $\mathfrak{L}(V)$ is a Lie algebra, with Lie bracket induced by
\[
\left[A_{[i]}, B_{[j]} \right] := \sum_{k\geq 0} {i \choose k} \left(A_{(k)}\cdot B \right)_{[i+j-k]}.
\]
The axiom on the vacuum vector $|0\rangle$ implies that $|0\rangle _{[-1]}$ is central in 
$\mathfrak{L}(V)$. 

One has a Lie algebra homomorphism $\mathfrak{L}(V)\rightarrow \textrm{End}(V)$: the element $A_{[i]}$ is mapped to the Fourier coefficient $A_{(i)}$ of the vertex operator $Y(A,z)=\sum_i A_{(i)}z^{-i-1}$. 
More generally, $\mathfrak{L}(V)$ is spanned by series of type $\sum_{i\geq i_0} f_i \, A_{[i]}$, for $A\in V$, $f_i\in\mathbb{C}$, and $i_0\in \mathbb{Z}$; the series $\sum_{i\geq i_0} f_i \, A_{[i]}$ maps to
\[
\textrm{Res}_{z=0}\,\, Y(A,z)\sum_{i\geq i_0} f_i z^i dz
\]
in $\textrm{End}(V)$ \cite[\S 4.1]{bzf}.
This defines an action of $\mathfrak{L}(V)$ on $V$.

Similarly to \S\ref{functors}, one can extend $\mathfrak{L}(V)$ to a functor of Lie algebras assigning to a $\mathbb{C}$-algebra $R$ the Lie algebra $( V\otimes R(\!(t)\!) ) / \textrm{Im}\, \partial$.

The Lie algebra $\mathfrak{L}(V)$ has a Lie subalgebra isomorphic to the Virasoro algebra: namely, the subalgebra generated by the elements 
\begin{equation} \label{eq:VirinLV}
c\cdot  |0\rangle _{[-1]} \cong K \quad\mbox{and}\quad \omega_{[p]} \cong L_{p-1}, \quad \mbox{for }p\in \mathbb{Z},
\end{equation}
where $|0\rangle$ and $\omega$ are the vacuum and the conformal vector, respectively, and $c$ is the central charge of $V$. 
Via the above identification and the axiom on the vacuum vector, the central element $K\in \mathfrak{L}(V)$ acts on $V$ as multiplication by $c$. The action of $\mathfrak{L}(V)$ on $V$  extends then the action of $\textrm{Vir}$ on $V$.

In particular, $\mathfrak{L}(V)$ has Lie subalgebras isomorphic to $\textrm{Der}\, \mathcal{O}$, $\textrm{Der}_0\, \mathcal{O}$, and $\textrm{Der}_+ \mathcal{O}$, as these are Lie subalgebras of the Virasoro algebra. 
   
Given a  $V$-module $M$, there is a Lie algebra homomorphism $\mathfrak{L}(V)\rightarrow \textrm{End}(M)$: the element $A_{[i]}$ is mapped to the Fourier coefficient $A^M_{(i)}$ of the vertex operator $Y^M(A,z)$.
This defines an action of $\mathfrak{L}(V)$ on $M$ extending the action of $\textrm{Vir}$ on $M$.

\subsection{Compatibility of actions of $\mathfrak{L}(V)$ and $\textup{\rm Aut}_+\mathcal{O}$ on $V$-modules}
\label{HCcompM}
The left action of $\textrm{Aut}_+\mathcal{O}$ on $M$ gives rise to a right action by $v\cdot \rho :=\rho^{-1}\cdot v$, for $\rho\in\textrm{Aut}_+\mathcal{O}$ and $v\in M$. 
The action of $\mathfrak{L}(V)$ on $M$ induces 
an anti-homomorphism of Lie algebras 
\[
\alpha_M \colon \mathfrak{L}(V) \rightarrow \textrm{End}(M)
\]
(that is, $-\alpha_M$ is a homomorphism of Lie algebras).
The \textit{anti-homomorphism} $\alpha_M$ is compatible with the \textit{right} action of $\textrm{Aut}_+\mathcal{O}$ on $M$ in the following sense:

\begin{enumerate}

\item the restriction of $\alpha_M$ to $\textrm{Der}_+ \mathcal{O}$ coincides with the differential of the right action of $\textrm{Aut}_+\mathcal{O}$ on $M$ (equivalently, the restriction of $-\alpha_M$ to $\textrm{Der}_+ \mathcal{O}$ coincides with the differential of the left action of $\textrm{Aut}_+\mathcal{O}$ on $M$);

\item for each $\rho\in \textrm{Aut}_+\mathcal{O}$, 
the following diagram commutes
\[
\begin{tikzcd}
\mathfrak{L}(V) \arrow[rightarrow]{r}{\alpha_M} \arrow[rightarrow, swap]{d}{\textrm{Ad}_\rho} & \textrm{End}(M) \arrow[rightarrow]{d}{\rho_*}\\
\mathfrak{L}(V) \arrow[rightarrow]{r}{\alpha_M} & \textrm{End}(M) .
\end{tikzcd}
\]
Here, $\textrm{Ad}$ is the adjoint representation of $\textrm{Aut}_+\mathcal{O}$ on $\mathfrak{L}(V)$ induced from
\[
\textrm{Ad}_\rho \left( \sum_{i\in \mathbb{Z}} A_{[i]}z^{-i-1} \right) = \sum_{i\in \mathbb{Z}} \left( \rho_z^{-1}\cdot A \right)_{[i]} \rho(z)^{-i-1}
\]
with $\rho_z(t):=\rho(z+t)-\rho(z)$ and $A \in V$.
Finally, $ \rho_* (-) :=\rho^{-1}\,(-)\,\rho$.

\end{enumerate}

Property (i) follows from the definition of the action of $\textrm{Aut}_+\mathcal{O}$ on $M$ as integration of the action of $\textrm{Der}_+\mathcal{O}$ on $M$. 
Property (ii) is due to Y.-Z.~Huang \cite[Prop.~7.4.1]{yzhuang}, \cite[\S 17.3.13]{bzf}:
\[
\rho^{-1}Y^M(A,z) \rho = Y^M( \rho_z^{-1}\cdot A, \rho(z)), \quad \mbox{for $A\in V$}.
\]


\section{Virasoro uniformization for stable coordinatized curves}
\label{AutBundles}
We present here the \textit{Virasoro uniformization}. The statement is in \cite{tuy}, following prior work in \cite{adkp}, \cite{besh}, \cite{kvir}. For completeness, we prove the result which we use here, extending to families of stable curves with singularities the argument for families of smooth curves given in \cite{bzf}.

Let $\widetriangle{\mathcal{M}}_{g,n}$ be the moduli space  parametrizing objects  $(C, P_\bullet=(P_1,\ldots,P_n), t_\bullet=(t_1,\ldots,t_n))$,  where $(C,P_\bullet)$ is a stable $n$-pointed curve of genus $g$, and $t_i$ is a formal coordinate at $P_i$. 
Let $(\mathcal{C}\rightarrow S, P_\bullet)$ be a versal family of stable pointed curves over a smooth base $S$, together with $n$ sections $P_i\colon S \rightarrow \mathcal{C}$ and formal coordinates~$t_i$ defined in a formal neighborhood of $P_i(S)\subset \mathcal{C}$. 
These data give rise to a moduli map $S\rightarrow \widetriangle{\mathcal{M}}_{g,n}$.
Assume that each irreducible component of $C_s$ contains at least one marked point $P_i(s)$, for all $s\in S$. This ensures that $\mathcal{C}\setminus P_\bullet(S)$ is affine. Let $\Delta$ be the divisor of singular curves in $S$. Here and throughout, $\mathcal{T}_{S}(-\log \Delta)$ is the sheaf of vector fields on $S$ preserving $\Delta$. 

\begin{theorem}[Virasoro uniformization \cite{adkp}, \cite{besh}, \cite{kvir}, \cite{tuy}]
\label{Virunif}
With notation as above, there exists an  anti-homo\-mor\-phism of Lie algebras 
\[
\alpha\colon \left(\textup{Der}\, \mathcal{K}(\mathbb{C})\right)^{\oplus n} \,\widehat\otimes_\mathbb{C}\, H^0(S,\mathcal{O}_S)\rightarrow H^0\left(S, \mathcal{T}_{S}(-\log \Delta)\right)
\]
(that is, $-\alpha$ is a homomorphism of Lie algebras)
extending $\mathcal{O}_{S}$-linearly to a surjective anti-homo\-mor\-phism of Lie algebroids (see \S \ref{LieAlgebroids}), called the \textit{anchor map}
\begin{equation}
\label{anchor}
\begin{tikzcd}
a\colon \left(\textup{Der}\, \mathcal{K}(\mathbb{C})\right)^{\oplus n} \,\widehat\otimes_\mathbb{C}\, \mathcal{O}_{S} \arrow[two heads]{r} & \mathcal{T}_{S} (-\log \Delta).
\end{tikzcd}
\end{equation}
The action of $(\textup{Der}\, \mathcal{K})^n$ on $\widetriangle{\mathcal{M}}_{g,n}$ induced from $\alpha$ is compatible with the action of $(\textup{Aut}\,\mathcal{O})^n$ on the fibers of $\widetriangle{\mathcal{M}}_{g,n}\rightarrow \overline{\mathcal{M}}_{g,n}$. The kernel of the anchor map $a$ is the subsheaf whose fiber at a point $(C, P_\bullet, t_\bullet)$ in $S$ is the Lie algebra $\mathscr{T}_C(C\setminus P_\bullet)$ of regular vector fields on $C\setminus P_\bullet$. 
\end{theorem}

Here and throughout, $\widehat\otimes$ is the completion of the usual tensor product with respect to the $t$-adic topology of $\textrm{Der}\, \mathcal{K}(\mathbb{C})=\mathbb{C}(\!(t)\!)\partial_t$. In particular, one has
\[
\left(\textup{Der}\, \mathcal{K}(\mathbb{C})\right)^{\oplus n} \,\widehat\otimes_\mathbb{C}\, H^0(S,\mathcal{O}_S) =  
\left(\textup{Der}\, \mathcal{K}\left(H^0(S,\mathcal{O}_S)\right)\right)^{\oplus n}
\] 
and
$ \left(\textup{Der}\, \mathcal{K}(\mathbb{C})\right)^{\oplus n} \,\widehat\otimes_\mathbb{C}\, \mathcal{O}_{S} = \oplus_{i=1}^n \mathcal{O}_{S}(\!( t_i)\!)\partial_{t_i}$.

Our first application of Theorem \ref{Virunif} is to show that sheaves  of $V$-modules  over a curve 
 carry a natural flat logarithmic connection (Propositions \ref{propMCJconnection} and \ref{propVconnection}). 
The connection for the sheaf $\mathscr{V}_{C}$ will be used  in the definition of the Lie algebra $\mathscr{L}_{C\setminus P_\bullet}(V)$ in \eqref{UCminusPV}. 
These results were known for smooth curves.

The structure of $\widetriangle{\mathcal{M}}_{g,n}$ is summarized in Figure \ref{fig:Moduliofcoordinatizedcurves} (see \S \ref{triangleMdef}).  The proof of Theorem \ref{Virunif} is contained in the remaining part of the section: the map $\alpha$ is described in \S \ref{AlphaDes}, compatibility  in \S \ref{Compat}.  
The space $\widetriangle{\mathcal{M}}_{g,n}$ is a principal $({\textrm{Aut}}\,\mathcal{O})^n$-bundle over $\overline{\mathcal{M}}_{g,n}$.
We will also need the intermediate principal bundle $\overline{\mathcal{J}}_{g,n}^{1,\times}$, which is the moduli space of elements of type $(C, P_\bullet, \tau_\bullet=(\tau_1,\dots,\tau_n))$ such that $\tau_i$ is a non-zero $1$-jet of a formal coordinate at~$P_i$ (see \S \ref{JetModuli}). 
In \S \ref{VUProof}, we describe the restriction of the anchor map to a curve $C$ and discuss the uniformization of $\mathscr{A}ut_C$, the fiber of  $\widetriangle{\mathcal{M}}_{g,1}\rightarrow \overline{\mathcal{M}}_{g}$ over $C$.  

Theorem \ref{Virunif} is extended by Theorem \ref{nVirAt}, which provides an action of the Virasoro algebra on the Hodge line bundle on $\widetriangle{\mathcal{M}}_{g,n}$. This will be a key ingredient in the proof of the main Theorem.

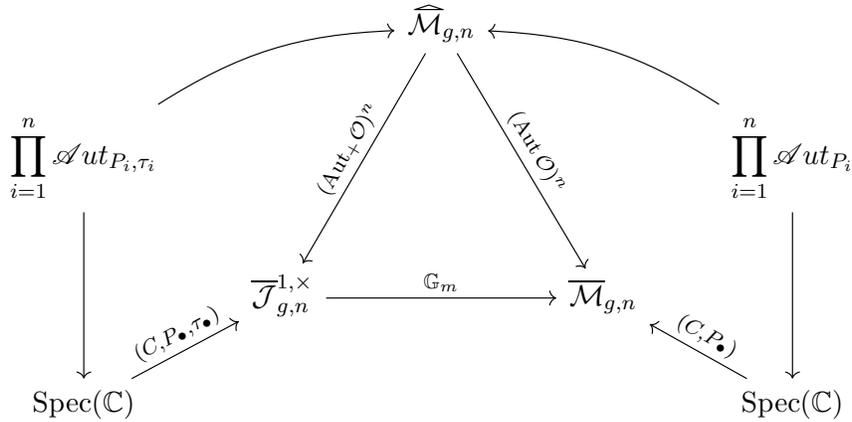
\begin{figure}[htb]
\centering
\begin{tikzcd}
{}&& \widetriangle{\mathcal{M}}_{g,n} \ar[rightarrow]{ddl}[sloped, pos=0.5]{(\textrm{Aut}_+\mathcal{O})^n} \ar[rightarrow]{ddr}[sloped, pos=0.5]{(\textrm{Aut}\, \mathcal{O})^n} &&\\
{\displaystyle \prod_{i=1}^n}\,\mathscr{A}ut_{P_i,\tau_i} \ar[rightarrow]{dd} \ar[rightarrow, bend left=15]{urr} &&&& {\displaystyle \prod_{i=1}^n}\,\mathscr{A}ut_{P_i} \ar[rightarrow]{dd} \ar[rightarrow, bend left=-15]{ull}\\
{}& \overline{\mathcal{J}}_{g,n}^{1,\times} \ar[rightarrow]{rr}{\mathbb{G}_m} && \overline{\mathcal{M}}_{g,n} &\\
\text{Spec}(\mathbb{C})\ar[rightarrow]{ur}[sloped, pos=0.5]{(C,P_\bullet,\tau_\bullet)}  &&&& \text{Spec}(\mathbb{C})  \ar[rightarrow]{ul}[sloped, pos=0.5]{ (C,P_\bullet)}
\end{tikzcd}
   \caption{The structure of moduli spaces of coordinatized curves.}
   \label{fig:Moduliofcoordinatizedcurves}
\end{figure}

\subsection{Lie algebroids}\label{LieAlgebroids} Following \cite{bb}, we briefly review the definition 
of a Lie algebroid, referred to here.  A logarithmic version is used in \S \ref{projconnCJ} and \S \ref{ProjConnVM}.  Let $S$ be a scheme over $\mathbb{C}$.  A \textit{Lie algebroid} $\mathcal{A}$ over $S$ is a quasi-coherent $\mathcal{O}_S$-module together with a $\mathbb{C}$-linear bracket $[\cdot , \cdot ] \colon \mathcal{A}\otimes_{\mathbb{C}} \mathcal{A} \rightarrow  \mathcal{A}$ and an  $\mathcal{O}_S$-module homomorphism $a \colon \mathcal{A}  \rightarrow \mathscr{T}_S$ called the \textit{anchor map}, for which: (i)
$a$ is a Lie algebra homomorphism, and (ii) $[x_1,fx_2]=f[x_1,x_2]+(a(x_1)\cdot f)x_2$, for $x_1$ and $x_2$ in $\mathcal{A}$ and $f \in \mathcal{O}_S$.  
The tangent sheaf with the identity anchor map is the simplest example of a Lie algebroid.

\subsection{Description of $\widetriangle{\mathcal{M}}_{g,n}$}
\label{triangleMdef}
The space $\widetriangle{\mathcal{M}}_{g,n}$ is a principal $({\textrm{Aut}}\,\mathcal{O})^n$-bundle over $\overline{\mathcal{M}}_{g,n}$.
Indeed, consider the forgetful map  $\widetriangle{\mathcal{M}}_{g,n}\rightarrow\overline{\mathcal{M}}_{g,n}$. When $n=1$, the fiber over a $\mathbb{C}$-point $(C,P)$ of $\overline{\mathcal{M}}_{g,1}$ 
 is the set of formal coordinates at~$P$, i.e.,  
\[
\mathscr{A}ut_P :=\left\{t \in \widehat{\mathscr{O}}_P \, | \, t(P)=0, \, (dt)_P\not= 0 \right\}.
\]
Here $\widehat{\mathscr{O}}_P$ is the completed local ring at the point $P$; after choosing a formal coordinate $t$ at $P$, one has $\widehat{\mathscr{O}}_P\cong \mathbb{C}\llbracket t\rrbracket$.
The set $\mathscr{A}ut_P$ admits a simply transitive right action of the group $\textrm{Aut}\,\mathcal{O}(\mathbb{C})$ by change of coordinates: 
\[
\mathscr{A}ut_P \times \textrm{Aut}\,\mathcal{O}(\mathbb{C}) \rightarrow \mathscr{A}ut_P, \qquad (t, \rho) \mapsto t\cdot \rho := \rho(t).
\]
Elements of $\textrm{Aut}\,\mathcal{O}(\mathbb{C})$ are power series $a_1 z+ a_2 z^2+\cdots$, such that $a_1\not=0$, and the group law is the composition of series: $\rho_1\cdot \rho_2= \rho_2\circ \rho_1$. Thus, $\mathscr{A}ut_P$ is an $(\textrm{Aut}\,\mathcal{O})$-torsor over a point.
More generally, $\widetriangle{\mathcal{M}}_{g,n}$ is an $(\textrm{Aut}\,\mathcal{O})^n$-torsor over  $\overline{\mathcal{M}}_{g,n}$.

\subsubsection{Description of $\alpha$}\label{AlphaDes}
The map $\alpha$ can be constructed as the differential of the right action of the group ind-scheme $(\textrm{Aut}\, \mathcal{K})^n$ on $\widetriangle{\mathcal{M}}_{g,n}$ preserving the divisor $\Delta$ (as in \cite[\S 17.3.4]{bzf}).
As  differential of a right action, the map $\alpha$ is an anti-homomorphism of Lie algebras, that is, 
$\alpha([l,m]) =$\allowbreak $-[\alpha(l), \alpha(m)]$.

The action is constructed as follows. The group ind-scheme $(\textrm{Aut}\, \mathcal{K})^n$ acts naturally on the punctured formal discs of the $n$ marked points on  stable $n$-pointed curves. Given a stable $n$-pointed curve $(C,P_\bullet)$ and an element $\rho\in (\textrm{Aut}\, \mathcal{K})^n$, 
one glues $C\setminus P_\bullet$ and the formal discs of the $n$ marked points by twisting with the automorphism $\rho$ of the $n$ punctured formal discs to obtain a stable pointed curve. 
The resulting curve has the same topological type of the starting curve; in particular, the action preserves the divisor $\Delta$.
This description of the action can be carried out in families, as in  \cite[\S 17.3.4]{bzf}.
As each irreducible component of a pointed curve $(C,P_\bullet)$ is assumed to have at least one marked point, the open curve $C\setminus P_\bullet$ is affine, as are the formal discs at the marked points.
Since smooth affine varieties have no non-trivial infinitesimal deformations, all infinitesimal deformations are obtained by the differential of the action of $(\textrm{Aut}\, \mathcal{K})^n$, hence the surjectivity of the anchor map $a$.

The map $\alpha$ also follows from a canonical map from $(\textrm{Der}\, \mathcal{K})^n$ to the space of tangent directions preserving $\Delta$ at any point $(C,P_\bullet, t_\bullet)$ in $\widetriangle{\mathcal{M}}_{g,n}$. The Lie subalgebra $(\textrm{Der}\, \mathcal{O})^n$ is canonically isomorphic to the  space of tangent directions preserving the nodes of $C$ along the fiber of the forgetful map $\widetriangle{\mathcal{M}}_{g,n}\rightarrow \overline{\mathcal{M}}_{g}$, as in \S\ref{alphaconstr}. Finally, 
the vector fields $t_i^p\partial_{t_i}$ with $p<0$ on the punctured disk around the point $P_i$ having a pole at $P_i$ correspond to  infinitesimal changes of the complex structure on the curve~$C$  preserving the topological type of $C$ (hence preserving $\Delta$).

\subsubsection{}
\label{Compat}
\label{HCcomp} 
The  action of $(\textrm{Der}\, \mathcal{K})^n$ on $\widetriangle{\mathcal{M}}_{g,n}$ via $\alpha$ is compatible with the right action of $(\textrm{Aut}\,\mathcal{O})^n$ along the fibers of 
$\widetriangle{\mathcal{M}}_{g,n} \rightarrow \overline{\mathcal{M}}_{g,n}$, that is:

\begin{enumerate}
\item the restriction of $\alpha$ to the Lie subalgebra  $(\textrm{Der}_0\,\mathcal{O})^n=\textrm{Lie}((\textrm{Aut}\,\mathcal{O})^n)$ of $(\textrm{Der}\, \mathcal{K})^n$ coincides with the differential of the right action of $(\textrm{Aut}\,\mathcal{O})^n$ along the fibers of the principal $(\textrm{Aut}\,\mathcal{O})^n$-bundle 
$\widetriangle{\mathcal{M}}_{g,n} \rightarrow \overline{\mathcal{M}}_{g,n}$; 
\item for each $\rho\in (\textrm{Aut}\,\mathcal{O})^n$, the following diagram commutes
\[
\begin{tikzcd}
\left(\textup{Der}\, \mathcal{K}(\mathbb{C})\right)^{\oplus n} \,\widehat\otimes_\mathbb{C}\, H^0(S,\mathcal{O}_S) \arrow[rightarrow]{r}{\alpha} \arrow[rightarrow, swap]{d}{\textrm{Ad}_\rho} & H^0\left(\widetriangle{\mathcal{M}}_{g,n}, \mathcal{T}_{\widetriangle{\mathcal{M}}_{g,n}} (-\log \Delta)\right) \arrow[rightarrow]{d}{\rho_*}\\
\left(\textup{Der}\, \mathcal{K}(\mathbb{C})\right)^{\oplus n} \,\widehat\otimes_\mathbb{C}\, H^0(S,\mathcal{O}_S) \arrow[rightarrow]{r}{\alpha} & H^0\left(\widetriangle{\mathcal{M}}_{g,n}, \mathcal{T}_{\widetriangle{\mathcal{M}}_{g,n}}(-\log \Delta)\right) .
\end{tikzcd}
\]
That is,
$\alpha (\textrm{Ad}_\rho (\cdot)) = \rho_* (\alpha(\cdot))$.
Here, $\textrm{Ad}$ is the adjoint representation of $(\textrm{Aut}\,\mathcal{O})^n$ on $(\textrm{Der}\, \mathcal{K})^n$. 
Moreover, $ \rho_* (\alpha(\cdot)) :=\rho^{-1}\, \alpha(\cdot)\,\rho$.
\end{enumerate}

\subsection{The moduli space $\overline{\mathcal{J}}_{g,n}^{1,\times}$}\label{JetModuli}
A $1$-jet at  a smooth point $P$ on a stable curve is a cotangent vector at $P$. This can be seen as an equivalence class of functions for the relation: $\tau\sim \sigma$ if and only if $\tau-\sigma \in \mathfrak{m}_P^2$,
for $\tau$ and $\sigma \in \widehat{\mathscr{O}}_P$, where $\mathfrak{m}_P$ is the maximal ideal of $\widehat{\mathscr{O}}_P$.  We say that $\tau$ is the $1$-jet of $t\in \widehat{\mathscr{O}}_P$ if $\tau$ is the equivalence class represented by $t$.
Let
\[
\mathscr{A}ut_{P,\tau}:=\{t \in \mathscr{A}ut_P \, | \, \tau \mbox{ is the $1$-jet of t} \}.
\]
This is an $(\textrm{Aut}_+\mathcal{O})$-torsor over a point, with  $\textrm{Aut}_+\mathcal{O}$ acting on the right by change of coordinates.
Recall that $\textrm{Aut}_+\mathcal{O}$ is the subgroup of $\textrm{Aut}\,\mathcal{O}$ of elements $\rho(z)=z+a_2 z^2+\cdots$.  One can show that $\textrm{Aut}\,\mathcal{O}= \mathbb{G}_m \ltimes \textrm{Aut}_+\mathcal{O}$.

Let  $\overline{\mathcal{J}}_{g,n}^{1,\times}$ be the moduli space of objects of type $(C, P_\bullet, \tau_\bullet)$, where $(C, P_\bullet)$ is an $n$-pointed, genus $g$ stable curve and $\tau_\bullet=(\tau_1,\dots,\tau_n)$ with each $\tau_i$  a non-zero $1$-jet of a formal coordinate at~$P_i$. The space $\overline{\mathcal{J}}_{g,n}^{1,\times}$ is a principal $(\mathbb{G}_m)^n$-bundle over $\overline{\mathcal{M}}_{g,n}$. 
Let $\Psi_i$ be the cotangent line bundle on $\overline{\mathcal{M}}_{g,n}$ corresponding to the $i$-th marked point.
We identify $\overline{\mathcal{J}}_{g,n}^{1,\times}$ with the product of the principal $\mathbb{C}^\times$-bundles $\Psi_i\setminus \{\mbox{zero section}\}$ over $\overline{\mathcal{M}}_{g,n}$, for $i=1,\dots,n$.

There is a natural map $\widetriangle{\pi}\colon\widetriangle{\mathcal{M}}_{g,n}\rightarrow \overline{\mathcal{J}}_{g,n}^{1,\times}$ obtained by mapping each local coordinate to its $1$-jet. This realizes $\widetriangle{\mathcal{M}}_{g,n}$ as a principal $(\textrm{Aut}_+\mathcal{O})^n$-bundle over $\overline{\mathcal{J}}_{g,n}^{1,\times}$. 
The  action of $(\textrm{Der}\, \mathcal{K})^n$ on $\widetriangle{\mathcal{M}}_{g,n}$ from \eqref{anchor} 
is compatible with the action of $ (\textrm{Aut}_+\mathcal{O})^n$ along the fibers of~$\widetriangle{\pi}$, as in \S\ref{HCcomp}.

\subsection{Coordinatized curves}
\label{CC}
Consider the composition of the projection $\widetriangle{\mathcal{M}}_{g,1}\rightarrow \overline{\mathcal{M}}_{g,1}$
with the forgetful morphism $\overline{\mathcal{M}}_{g,1} \to \overline{\mathcal{M}}_g$. We  define $\mathscr{A}ut_C$ as the fiber of the map $\widetriangle{\mathcal{M}}_{g,1}\rightarrow \overline{\mathcal{M}}_{g}$ over the point in $\overline{\mathcal{M}}_g$ corresponding to $C$ (Figure \ref{fig:AutCCJ}).
The bundle $\mathscr{A}ut_C$ is a principal $(\textrm{Aut}\,\mathcal{O})$-bundle on $C$ whose fiber at a smooth $P\in C$ is $\mathscr{A}ut_{P}$. 
The description of the fiber over a nodal point $P$ goes as follows: We replace the pair $(C,P)$  with its stable reduction $(C',P')$ and let the fiber of $\mathscr{A}ut_C$ over $P$ be the space $\mathscr{A}ut_{P'}$. Equivalently, if $P$ is a node, we replace $C$ with its partial normalization $C^N$ at $P$ and consider the two points $P_+$ and $P_-$ lying above $P$. The fibers over $P_+$ and $P_-$ are the spaces $\mathscr{A}ut_{P_+}$ and $\mathscr{A}ut_{P_-}$. 
If  locally  near the node $P$ the curve $C$ is given by the equation $s_+ s_- =0$, with $s_\pm$ a formal coordinate at $P_\pm$, we glue $\mathscr{A}ut_{P_+}$ and $\mathscr{A}ut_{P_-}$ via the isomorphism
\begin{equation}
\label{eq:gluingAut}
\mathscr{A}ut_{P_+}\cong \textrm{Aut}\,\mathcal{O} \xrightarrow{\cong} \textrm{Aut}\,\mathcal{O} \cong\mathscr{A}ut_{P_-}, \qquad \rho(s_+)\mapsto \rho\circ \gamma(s_-),
\end{equation}
where $\gamma\in \textrm{Aut}\,\mathcal{O}$ is 
\[
\gamma(z):= \frac{1}{1+z}-1 = -z +z^2 -z^3 + \cdots.
\]
The identification of $\mathscr{A}ut_{P_+}$ and $\mathscr{A}ut_{P_-}$ via this isomorphism defines the fiber of $\mathscr{A}ut_C$ at the node $P$.
The geometry motivating the choice of the  gluing isomorphism is explained in \cite[\S 2.2.2]{DGT2}.
The bundle $\mathscr{A}ut_C$ is locally trivial in the Zariski topology.

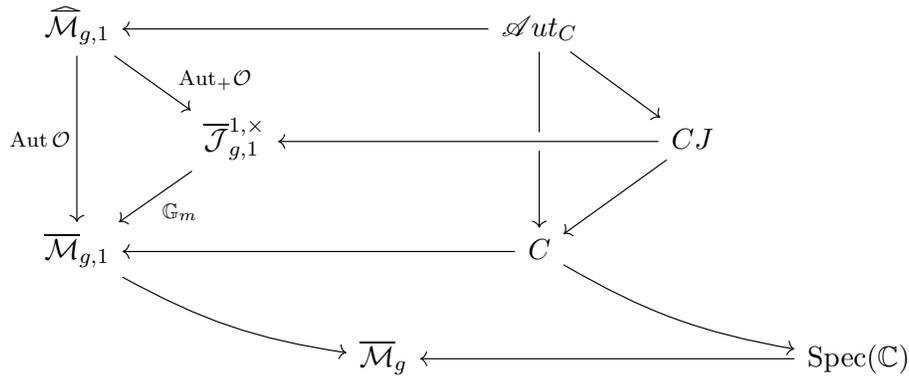
\begin{figure}[htb] 
\centering
\begin{tikzcd}
\widetriangle{\mathcal{M}}_{g,1} \ar[rightarrow, swap]{dd}{\textrm{Aut}\,\mathcal{O}} \ar[rightarrow, near end]{dr}{\textrm{Aut}_+\mathcal{O}} &&& \mathscr{A}ut_C \ar[rightarrow]{lll}  \ar[rightarrow]{dd} \ar[rightarrow]{dr} &&\\
& \overline{\mathcal{J}}_{g,1}^{1,\times} \ar[rightarrow]{dl}{\mathbb{G}_m} &&& CJ \ar[crossing over, rightarrow]{lll} \ar[rightarrow]{dl} &\\
\overline{\mathcal{M}}_{g,1} \ar[rightarrow, bend right=10]{drr} &&& C \ar[rightarrow]{lll} \ar[rightarrow, bend right=10]{drr} &&\\
&& \overline{\mathcal{M}}_{g} &&& \textrm{Spec}(\mathbb{C}) \ar[rightarrow]{lll}
\end{tikzcd}
   \caption{The definition of $\mathscr{A}ut_C$ and $CJ$. The curve $C$ is identified with the fiber of $\overline{\mathcal{M}}_{g,1}\rightarrow \overline{\mathcal{M}}_g$ over the moduli point $[C]\in \overline{\mathcal{M}}_g$.}
   \label{fig:AutCCJ}
\end{figure}

Similarly, let $CJ$ be the fiber of the forgetful map $\overline{\mathcal{J}}_{g,1}^{1,\times}\rightarrow \overline{\mathcal{M}}_{g}$ over the point in $\overline{\mathcal{M}}_g$ corresponding to the curve $C$. 
The space $CJ$ is a principal $\mathbb{G}_m$-bundle on $C$, whose points are pairs $(P,\tau)$, where $P$ is a point in the fiber of $\overline{\mathcal{M}}_{g,1}\rightarrow \overline{\mathcal{M}}_g$ over the moduli point $[C]\in \overline{\mathcal{M}}_g$, and $\tau$ is a non-zero $1$-jet of a formal coordinate at $P$.
Mapping a formal coordinate to its $1$-jet realizes $\mathscr{A}ut_C$ as a principal $(\textrm{Aut}_+ \mathcal{O})$-bundle on $CJ$, whose fibers at $(P,\tau)$ is $\mathscr{A}ut_{P,\tau}$, as pictured in Figure \ref{fig:AutBundles}. 

\begin{figure}[htb] 
\centering
\begin{tikzcd}
{}&& \mathscr{A}ut_C \ar[rightarrow]{ddl}[sloped, pos=0.5]{\textrm{Aut}_+\mathcal{O}} \ar[rightarrow]{ddr}[sloped, pos=0.5]{\textrm{Aut}\, \mathcal{O}} &&\\
\mathscr{A}ut_{P,\tau} \ar[rightarrow]{dd} \ar[rightarrow, bend left=15]{urr} &&&& \mathscr{A}ut_{P} \ar[rightarrow]{dd} \ar[rightarrow, bend left=-15]{ull}\\
{}& CJ \ar[rightarrow]{rr}{\mathbb{G}_m} && C &\\
(P,\tau) \ar[hookrightarrow]{ur} \ar[|->]{rrrr} &&&& P \ar[hookrightarrow]{ul}
\end{tikzcd}
   \caption{The structure of $\mathscr{A}ut_C$.}
   \label{fig:AutBundles}
\end{figure}
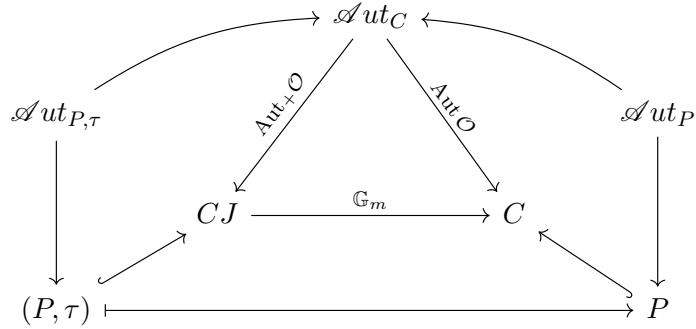

\subsection{Uniformization of $\mathscr{A}ut_C$}\label{VUProof}
\label{alphaC} 
As described in \cite[\S 17.1]{bzf}, the Lie algebra $\textrm{Der}\, \mathcal{O}$ has a simply transitive  action on the space $\mathscr{A}ut_C$ over a smooth curve~$C$. Here we discuss the action, and generalize it to the case of stable curves. Given a stable curve $C$, let $D$ be the divisor in $\mathscr{A}ut_C$
lying over the singular locus of $C$. Let $\mathscr{T}_{\mathscr{A}ut_C}(-\log D)$ be the sheaf of $\mathbb{C}$-linear derivations of $\mathscr{A}ut_C$ which preserve the ideal defining $D$. There is an anti-homomorphism of Lie algebras
\begin{equation}
\label{alphaCmap}
\alpha_C\colon \textrm{Der}\, \mathcal{O}(\mathbb{C}) \,\widehat\otimes_\mathbb{C}\, H^0(\mathscr{A}ut_C,\mathscr{O}_{\mathscr{A}ut_C})
\rightarrow H^0\left(\mathscr{A}ut_C, \mathscr{T}_{\mathscr{A}ut_C}(-\log D)\right) 
\end{equation}
described below in \S \ref{alphaconstr},
that extends $\mathscr{O}_{\mathscr{A}ut_C}$-linearly to an anti-isomorphism of Lie algebroids
\begin{equation}
\label{anchorAutC}
a_C\colon \textrm{Der}\, \mathcal{O}(\mathbb{C}) \,\widehat{\otimes}_\mathbb{C}\, \mathscr{O}_{\mathscr{A}ut_C}\xrightarrow{\sim} \mathscr{T}_{\mathscr{A}ut_C}(-\log D).
\end{equation}
The restrictions of the maps $\alpha$ and $a$ from Theorem \ref{Virunif} to the fiber of the forgetful map $\widetriangle{\mathcal{M}}_{g,1} \rightarrow \overline{\mathcal{M}}_{g}$ over a point $C$ in $\overline{\mathcal{M}}_{g}$ coincide with the maps $\alpha_C$ and $a_C$, respectively.

\subsubsection{Description of $\alpha_{C}$}
\label{alphaconstr}
A heuristic description of the map $\alpha_C$ as a bijection of vector spaces can be given by showing how to each element of $\textrm{Der}\, \mathcal{O}$ corresponds a vector field on $\mathscr{A}ut_C$ preserving the divisor $D$. Since a vector field is the choice of a tangent vector at each point, it is enough to assign to each element of $\textrm{Der}\, \mathcal{O}$ a tangent vector preserving $D$ at each point in $\mathscr{A}ut_C$. 

Consider the following diagrams: 
\[
\begin{tikzcd}
\mathscr{A}ut_C  \ar{d}{p} && \textrm{Aut}_+\mathcal{O} \ar{d} \ar{ll}[swap]{f}\\
CJ  && \text{Spec}(\mathbb{C}) \ar{ll}, 
\end{tikzcd}
\qquad \qquad 
\begin{tikzcd}
CJ  \ar{d}{q}  && \mathbb{G}_m \ar{d} \ar{ll}[swap]{h}\\
C && \text{Spec}(\mathbb{C}) \ar{ll}. 
\end{tikzcd}
\]
Since $\mathscr{A}ut_C$ (resp.,~$CJ$) is a principal $(\textrm{Aut}_+\mathcal{O})$-bundle (resp.,~$\mathbb{G}_m$-bundle), locally the two diagrams are cartesian. This implies that we can decompose the tangent sheaf of $\mathscr{A}ut_C$ and $CJ$ as the following direct sums
\begin{align*}
\mathscr{T}_{\mathscr{A}ut_C}(-\log D)  &= p^*\mathscr{T}_{CJ}(-\log D)  \oplus f_*\,\mathscr{T}_{\textrm{Aut}_+\mathcal{O}},\\
\mathscr{T}_{CJ}(-\log D)  &= q^*\mathscr{T}_C(-\log D)  \oplus h_*\,\mathscr{T}_{\mathbb{G}_m}.
\end{align*}
Observe that the elements of $\mathscr{T}_C$ are derivations of $\mathscr{O}_C$ which preserve the ideal defining the singular points, so we can rewrite the above equality as \begin{align*}
\mathscr{T}_{CJ}(-\log D)  &= q^*\mathscr{T}_C  \oplus h_*\,\mathscr{T}_{\mathbb{G}_m}.
\end{align*}

Combining the two decompositions, we obtain
\[
\mathscr{T}_{\mathscr{A}ut_C}(-\log D) = p^*q^*\mathscr{T}_C \oplus p^*h_*\,\mathscr{T}_{\mathbb{G}_m} \oplus f_*\,\mathscr{T}_{\textrm{Aut}_+\mathcal{O}}.
\]
From this characterization, it follows that the tangent space of $\mathscr{T}_{\mathscr{A}ut_C}$ at a point $(P,t)$, where $P$ is a smooth point of $C$, can be described as the sum 
\[
T_{C,P} \oplus T_{\mathbb{G}_m}\oplus T_{\textrm{Aut}_+\mathcal{O}} 
\]
which is isomorphic to 
\begin{equation}
\label{DecofTang}
\mathbb{C}\partial_t \oplus \mathbb{C}t \partial_t \oplus \textrm{Der}_+\mathcal{O}(\mathbb{C})= \text{Der}\,\mathcal{O}(\mathbb{C}) .
\end{equation}
The space $\mathbb{C} t\partial_t$ corresponds to infinitesimal changes of the $1$-jet $\tau$ of $t$ fixing the point $P$. In the same spirit, the space $\mathbb{C} \partial_t$ is identified with the tangent direction at $(P,t)$ corresponding to infinitesimal changes of the point $P$ on the curve.

Note that when $P$ is not smooth, we can replace $(C,P)$ with its stable reduction $(C',P')$ where $P'$ lies in a rational component of $C'$. Also in this case we obtain that the space is isomorphic to \eqref{DecofTang}.
In this case, however, infinitesimal changes of the point $P'$ on the rational component are identified by the automorphisms of the rational component, hence the tangent space $\mathbb{C} \partial_t$ is zero at $P'$. It follows that all tangent directions described above preserve the singular locus $D$ in $\mathscr{A}ut_C$. 

From this identification of $\textrm{Der}\, \mathcal{O}$ with the space of tangent directions preserving $D$ at any  point in $\mathscr{A}ut_C$, an element of $\textrm{Der}\, \mathcal{O}$ gives rise to a tangent vector at each point of $\mathscr{A}ut_C$, hence a vector field on $\mathscr{A}ut_C$ preserving $D$.

More precisely, the map $\alpha_C$ is given by taking the differential of the right action of the exponential of $\textrm{Der}\, \mathcal{O}$.
Recall that  the exponential of $\textrm{Der}\, \mathcal{O}$ is the group ind-scheme $\underline{\textrm{Aut}}\, \mathcal{O}$ (see \S\ref{functors}).
Consider the principal $(\underline{\textrm{Aut}}\, \mathcal{O})$-bundle 
\[
\underline{\mathscr{A}ut}\,_C :=\underline{\textrm{Aut}}\, \mathcal{O} \mathop{\times}_{\textrm{Aut}\, \mathcal{O}} \mathscr{A}ut_C
\]
on $C$. For a $\mathbb{C}$-algebra $R$, an $R$-point of $\underline{\mathscr{A}ut}\,_C$ is a pair $(P,t)$, where $P$ is an $R$-point of $C$, and $t$ is an element of $R\,\widehat{\otimes}\, \widehat{\mathscr{O}}_P$ such that there is a continuous isomorphism of algebras $R\,\widehat{\otimes}\,  \widehat{\mathscr{O}}_P \cong R\llbracket t \rrbracket$. The group ind-scheme $\underline{\textrm{Aut}}\, \mathcal{O}$ has a right action on $\underline{\mathscr{A}ut}\,_C$, and the differential of this action gives the map $\alpha_C$. The argument is presented in \cite[\S 17.1.3]{bzf} for smooth curves and generalizes to nodal curves. 


\section{$V$-module sheaves on curves and their  connections}
\label{Vmodbdles}
Here given a $V$-module $M$, we define sheaves  $\mathscr{M}_{C}$ over a curve $C$, and 
$\mathscr{M}_{CJ}$ over  $CJ$ (defined in \S\ref{CC}).    Considering $V$ as a module over itself, this gives rise to the vertex algebra sheaf $\mathscr{V}_{C}$ on a curve $C$.
We  show that
$\mathscr{M}_{CJ}$ and  $\mathscr{V}_{C}$ have  flat logarithmic connections (Propositions \ref{propMCJconnection} and \ref{propVconnection}). These results have been proved for smooth curves, and here we extend the results to stable curves with singularities using an instance of the Virasoro uniformization (Theorem \ref{Virunif}).

Proposition \ref{propVconnection} is used to define vector spaces of coinvariants and their sheaves  (see \S \ref{Coinvariants} and \S \ref{SheavesOfCoinvariants}). While one does not obtain a projectively flat connection on $V$-module sheaves in families, one does obtain a twisted logarithmic $\mathcal{D}$-module structure on certain quotients, giving rise to projectively flat connections on the sheaves of coinvariants (Theorems \ref{AtalgonCJ} and \ref{At2}).

The structure of a $V$-module bundle is summarized in Figure~\ref{figVmodbundles}.  
 
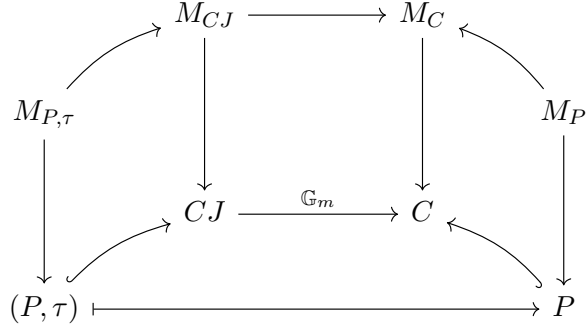
\begin{figure}[htb]
\centering
\begin{tikzcd}
{}& M_{CJ} \ar[rightarrow]{rr} \ar[rightarrow]{dd}{} && M_{C} \ar[rightarrow]{dd}{} &\\
M_{P,\tau} \ar[rightarrow]{dd} \ar[rightarrow, bend left=15]{ur} &&&& M_{P} \ar[rightarrow]{dd} \ar[rightarrow, bend left=-15]{ul}\\
{}& CJ \ar[rightarrow]{rr}{\mathbb{G}_m} && C &\\
(P,\tau) \ar[hookrightarrow, bend left=15]{ur} \ar[|->]{rrrr} &&&& P \ar[hookrightarrow, bend left=-15]{ul} 
\end{tikzcd}
   \caption{The structure of a $V$-module bundle.}
   \label{figVmodbundles}
\end{figure}

\subsection{The sheaf $\mathscr{M}_{CJ}$}
\label{Vmodbdle}
In \S \ref{Des} we describe the sheaves $\mathscr{M}_{CJ}$, and in \S \ref{conn} we define their flat logarithmic connection.

\subsubsection{Description}\label{Des}
Let $V$ be a vertex operator algebra, and $M=\oplus_{i\geq 0}M_i$ a  $V$-module. 
Let $C$ be a stable  curve. 
Consider the trivial vector bundle (of infinite rank when $M$ is infinite-dimensional)
\[ 
M_{\mathscr{A}ut_C} := \mathscr{A}ut_C \times \varinjlim_{i}\, M_{\leq i} = \varinjlim_{i}\,\mathscr{A}ut_C \times M_{\leq i}
\]
on $\mathscr{A}ut_C$. Here we still denote by $M_{\leq i}$ the finite-dimensional affine complex space associated with the $\mathbb{C}$-vector space $M_{\leq i}$.
 A $\mathbb{C}$-point of $M_{\mathscr{A}ut_C}$ is a triple $(P, t, m)$, where $(P,t)$ is a $\mathbb{C}$-point of the fiber of $\widetriangle{\mathcal{M}}_{g,1}\rightarrow \overline{\mathcal{M}}_{g}$ over the moduli point $[C]\in \overline{\mathcal{M}}_{g}$, and $m$ is an element of the module $M$. Observe that for an infinite-dimensional vector space $M$, the vector bundle $M_{\mathscr{A}ut_C}$ is not  an affine scheme over $\mathscr{A}ut_C$, but an ind-scheme. By abuse of notation, we will simply write
\[
M_{\mathscr{A}ut_C} := \mathscr{A}ut_C \times M.
\]
Its quasi-coherent sheaf of sections $\mathscr{M}_{\mathscr{A}ut_C}$  is locally free. 

As discussed in \S\ref{CC}, $\mathscr{A}ut_C$ is a principal $(\textrm{Aut}_+\mathcal{O})$-bundle over $CJ$, hence has a right action of $\textrm{Aut}_+\mathcal{O}$. 
The left action of $\textrm{Aut}_+\mathcal{O}$ on the modules $M_{\leq i}$, as in \S \ref{Action}, induces a right action of $\textrm{Aut}_+\mathcal{O}$ on $M_{\leq i}$, as in \S\ref{HCcompM}. 
It follows that $\mathscr{A}ut_C \times M_{\leq i}$ has an equivariant right action of $\textrm{Aut}_+\mathcal{O}$. This action is compatible with the inductive limit, so it induces an equivariant right action of $\textrm{Aut}_+\mathcal{O}$ on $M_{\mathscr{A}ut_C}$. 
On the $\mathbb{C}$-points of $M_{\mathscr{A}ut_C}$, the action of $\textrm{Aut}_+\mathcal{O}$ is:
\begin{align}
\label{Aut+Oaction} 
\begin{array}{ccc}
M_{\mathscr{A}ut_C}(\mathbb{C}) \times  \textrm{Aut}_+\mathcal{O}(\mathbb{C}) &\to & M_{\mathscr{A}ut_C}(\mathbb{C})\\
(P,t,m) \times \rho &\mapsto & \left(P, t\cdot\rho, \rho^{-1}\cdot m \right).
\end{array}
\end{align} 
The quotient of $M_{\mathscr{A}ut_C}$ by this action descends to a vector bundle on $CJ=\mathscr{A}ut_C/\textrm{Aut}_+\mathcal{O}$:
\[
M_{CJ} := \mathscr{A}ut_C \mathop{\times}_{\textrm{Aut}_+\mathcal{O}}  M.
\]
Recall that $p\colon \mathscr{A}ut_C \rightarrow CJ$ is the natural projection.  The sheaf of sections 
\[
\mathscr{M}_{CJ} := \left( M \otimes p_*\mathscr{O}_{\mathscr{A}ut_C} \right)^{\textrm{Aut}_+\mathcal{O}}
\]
of $M_{CJ}$ is then a locally free quasi-coherent $\mathscr{O}_{CJ}$-module. 

The fibers of $M_{CJ}$ on $CJ$ are described as follows. 
The fiber of $M_{\mathscr{A}ut_C}$ over a point $(P, \tau) \in CJ(\mathbb{C})$ is the trivial bundle $\mathscr{A}ut_{P,\tau} \mathop{\times} M$ on $\mathscr{A}ut_{P,\tau}$. The group $\textrm{Aut}_+\mathcal{O}$ acts equivariantly on $\mathscr{A}ut_{P,\tau} \mathop{\times} M$, as in \eqref{Aut+Oaction}. The fiber of $M_{CJ}$ over $(P, \tau) \in CJ(\mathbb{C})$ is the space
\begin{equation}
\label{MPtau}
M_{P,\tau} := \mathscr{A}ut_{P,\tau} \mathop{\times}_{\textrm{Aut}_+\mathcal{O}}  M
\end{equation}
defined as the quotient of $\mathscr{A}ut_{P,\tau} \mathop{\times} M$ modulo the relations
\begin{align*} 
(t \cdot\rho, m) = (t, \rho^{-1}\cdot m), \quad \mbox{for } \rho \in \textrm{Aut}_+\mathcal{O}(\mathbb{C}) \mbox{ and } (t,m)\in \mathscr{A}ut_{P,\tau} \mathop{\times} M.
\end{align*} 
One can identify $M_{P,\tau}$ with $\lbrace (\tau, m) \, | \, m \in M \rbrace \cong M$.
Thus we regard  $M_{P,\tau}$ as a realization of the module $M$ assigned at the pair $(P, \tau)$. It is independent of the curve~$C$.

\subsubsection{The flat logarithmic connection on $\mathscr{M}_{CJ}$} 
\label{conn} 
The strategy to produce a flat logarithmic connection on $\mathscr{M}_{CJ}$ is to first construct a flat logarithmic connection on the trivial bundle $\mathscr{A}ut_C {\times}  M$ on $\mathscr{A}ut_C$ induced by the action of the Lie algebra $\textrm{Der}\, \mathcal{O}$. Since the action of $\textrm{Der}\, \mathcal{O}$ is shown to be $(\textrm{Aut}_+\mathcal{O})$-equivariant in the sense of \S \ref{HCcomp}, the connection then descends to $\mathscr{M}_{CJ}$.
For a smooth curve, this has been treated in \cite[\S 17.1]{bzf}.

As discussed in \S\ref{HCcompM}, the action of $\mathfrak{L}(V)$ on $M$ gives rise to an anti-homomorphism of Lie algebras 
\[
\alpha_M \colon\textrm{Der}\, \mathcal{O}(\mathbb{C}) \rightarrow \textrm{End}(M). 
\]
The map $\alpha_M$ extends to an anti-homomorphism of sheaves of Lie algebras 
\begin{align}
\label{connonM}
\beta_C \colon \textrm{Der}\, \mathcal{O}(\mathbb{C}) \,\widehat{\otimes}_\mathbb{C}\,\mathscr{O}_{\mathscr{A}ut_C} \rightarrow \textrm{End}(M) \otimes_\mathbb{C}\mathscr{O}_{\mathscr{A}ut_C}
\end{align}
defined by
\begin{align*}
l\otimes f \mapsto \alpha_M(l)\otimes f + \textrm{id}_M\otimes (\alpha_C(l)\cdot f), \quad \mbox{for $l\in \textrm{Der}\, \mathcal{O}(\mathbb{C})$ and $f\in \mathscr{O}_{\mathscr{A}ut_C}$}.
\end{align*}
From \eqref{alphaCmap}, $\alpha_C$ maps an element of $\textrm{Der}\, \mathcal{O}$ to a vector field on $\mathscr{A}ut_C$, hence the action on regular functions above is by derivations.
Composing with \eqref{anchorAutC}, the map \eqref{connonM} gives rise to a homomorphism of sheaves of Lie algebras 
\begin{align}
\label{fllogconnAM}
\mathscr{T}_{\mathscr{A}ut_C}(-\log D) \rightarrow \textrm{End}(M) \otimes_\mathbb{C}\mathscr{O}_{\mathscr{A}ut_C}
\end{align} 
where the target is the sheaf of endomorphisms of $\mathscr{A}ut_C {\times}  M$, and $D$ is the divisor in $\mathscr{A}ut_C$ lying over the singular locus in $C$. 

The properties of the maps $\alpha_C$ and $\alpha_M$ described respectively in \S \ref{HCcomp} and  \S\ref{HCcompM} imply that the map \eqref{fllogconnAM} is $(\textrm{Aut}_+\mathcal{O})$-equivariant. This is one of the main ingredients to deduce the following statement:

\begin{proposition} 
\label{propMCJconnection} 
The sheaf $\mathscr{M}_{CJ}$ is naturally equipped with a flat logarithmic connection, i.e., a homomorphism of Lie algebras 
\begin{align}
\label{nablaT}
\mathscr{T}_{CJ}(-\log D) \rightarrow \textrm{End}\left(\mathscr{M}_{CJ}\right).
\end{align}
\end{proposition}

\proof  Consider the exact sequence involving the tangent sheaves of $\mathscr{A}ut_C$ and $CJ$:
\begin{equation} 
\label{TangentExSeq} 
0 \to \mathscr{T}_{\mathscr{A}ut_C}^{\text{vert}}\to \mathscr{T}_{\mathscr{A}ut_C}(-\log D)  \to p^*\mathscr{T}_{CJ}(-\log D)  \to 0.
\end{equation}
From the isomorphism $\mathscr{T}_{\mathscr{A}ut_C}(-\log D)   \cong \text{Der}\,\mathcal{O}(\mathbb{C}) \,\widehat{\otimes}_\mathbb{C}\, \mathscr{O}_{\mathscr{A}ut_C}$ and the fact that the elements $t^i\partial_t \otimes f$ map to zero in $p^*\mathscr{T}_{CJ}(-\log D) $ for $i \geq 2$, we deduce that 
\[ 
\mathscr{T}_{\mathscr{A}ut_C}^{\text{vert}}  \cong \text{Der}_+\mathcal{O}(\mathbb{C}) \,\widehat{\otimes}_\mathbb{C}\, \mathscr{O}_{\mathscr{A}ut_C}.
\]

In particular this implies that pushing forward the exact sequence \eqref{TangentExSeq} and taking $(\text{Aut}_+\mathcal{O})$-invariants, we can describe $\mathscr{T}_{CJ}(-\log D) $ as the quotient
\[ 
\left(\text{Der}\,\mathcal{O}(\mathbb{C}) \,\widehat{\otimes}_\mathbb{C}\, p_*\mathscr{O}_{\mathscr{A}ut_C}\right)^{\text{Aut}_+\mathcal{O}} /  \left(\text{Der}_+\mathcal{O}(\mathbb{C}) \,\widehat{\otimes}_\mathbb{C}\, p_*\mathscr{O}_{\mathscr{A}ut_C}\right)^{\text{Aut}_+\mathcal{O}}.
\]
In order to induce a connection on $\mathscr{M}_{CJ}$ from the one on $\mathscr{M}$, it is sufficient to prove that $(\text{Der}_+\mathcal{O}(\mathbb{C}) \,\widehat{\otimes}_\mathbb{C}\, \mathscr{O}_{\mathscr{A}ut_C})^{\textrm{Aut}_+\mathcal{O}}$ acts trivially on $\mathscr{M}_{CJ}$. By definition, the action of $\mathscr{T}_{\mathscr{A}ut_C}^{\text{vert}}$ on $M \otimes \mathscr{O}_{\mathscr{A}ut_C}$ is obtained by differentiating the natural action of $\text{Aut}_+\mathcal{O}$. Since $\mathscr{M}_{CJ}$ is given by the elements of $\mathscr{M}_{\mathscr{A}ut_C}$ on which $\text{Aut}_+\mathcal{O}$ acts as the identity, then the associated Lie algebra action will be trivial, concluding the argument.
\endproof

For the reader familiar with the localization of modules of Harish-Chandra pairs (\cite{bb}; see also \cite[\S 3]{bfm}),
the construction in this section is the result of the localization with respect to the Harish-Chandra pair $(\textrm{Der}\, \mathcal{O}, \textrm{Aut}_+\mathcal{O})$: the $(\textrm{Der}\, \mathcal{O}, \textrm{Aut}_+\mathcal{O})$-module $M$ is transformed into the logarithmic $\mathcal{D}$-module $\mathscr{M}_{CJ}$ on $CJ$.

\subsection{The sheaf $\mathscr{M}_C$}\label{DesV}
In \S \ref{DesM} we describe the sheaf $\mathscr{M}_{C}$, and in \S \ref{vab}, using the flat logarithmic connection on $\mathscr{V}_{CJ}$, we show that there is a flat logarithmic connection on $\mathscr{V}_{C}$.

\subsubsection{$\mathbb{C}^\times$-equivariance}
\label{DesM}
In general, the action of $L_0$ cannot be integrated to an action of $\mathbb{C}^\times$, unless the action of $L_0$ has only integral eigenvalues.
 
We obtain an action of $\mathbb{C}^\times$ in a different way.
Our assumption is that a $V$-module $M$ is $\mathbb{Z}$-graded, with gradation bounded from below. The $\mathbb{Z}$-gradation induces an action of $\mathbb{C}^\times$ on $M$: 
\begin{align}
\label{CactM}
z\cdot v:= z^{-\deg v} v, \qquad \mbox{for $z\in \mathbb{C}^\times$ and homogeneous $v\in M$}.
\end{align}
Assume for simplicity that $M$ is simple, and let $a$ be the conformal dimension of $M$. 
We claim that the given action of $\mathbb{C}^\times$ integrates the action of $L_0$ on ${M}_{CJ}$ after tensoring with $\omega_C^{\otimes -a}$. 
This allows one to define a $\mathbb{C}^\times$-equivariant action on ${M}_{CJ}\otimes \omega_C^{\otimes -a}\rightarrow CJ$.
As $q\colon CJ\rightarrow C$ is a principal $\mathbb{G}_m$-bundle on the curve $C$, the bundle ${M}_{CJ}\otimes \omega_C^{\otimes -a}$ descends to a bundle  on $C$, and we define:
\[
{M}_{{C}} := \left( q_*{M}_{CJ} \otimes \omega_C^{\otimes -a} \right)^{\mathbb{G}_m} \otimes \omega_C^{\otimes a}.
\]
We denote by $\mathscr{M}_{{C}}$ the sheaf of sections of ${M}_{{C}}$,
and  ${M}_{P}$  the fiber of $\mathscr{M}_{{C}}$ over a point $P\in C$.

To verify the claim, 
we note that locally on $CJ$, the sheaf $\mathscr{M}_{CJ}$ is free, hence isomorphic to $M \otimes \mathscr{O}_{CJ}$. Fix a point $P$ on a curve $C$ and let $R$ be the local ring of $C$ at $P$. Locally, $CJ$ is given by $\mathrm{Spec}(R[w, w^{-1}])$, and
 $\mathscr{M}_{CJ}$ is given by $M \otimes R[w, w^{-1}]$. 
After trivializing $\omega^{-a}$ locally as $w^a$, one has a local isomorphism: $\mathscr{M}_{CJ} = \mathscr{M}_{CJ} \otimes w^a$.
In particular, over a point, one has the isomorphism $M = M\otimes w^a$. The action of $\mathbb{C}^\times$ on $M$ given by \eqref{CactM} may be thought of as an action of $\mathbb{C}^\times$ on $M\otimes w^a$ via this isomorphism.

It remains then to verify that this action of $\mathbb{C}^\times$ integrates the action of $L_0$ on $M\otimes w^a$.
That is, we need to check that $L_0$ acts on an element $v \otimes w^a$ in $M\otimes w^a$ as multiplication by $\deg(v)$. For this, we note that
$L_0$ acts on $v$ as multiplication by $\deg(v)+a$; and $L_0 = -w \partial_w$ acts on $w^a$ as multiplication by $-a$. In summary, one has a multiplication by $\deg(v)$, as desired.

Note that when $L_0(v)=(\deg v)v$ for homogeneous $v\in M$ (e.g., $M=V$), the $L_0$-action on $M$ integrates to the $\mathbb{C}^\times$-action in \eqref{CactM} on $M$. In general, the eigenvalues of $L_0$ are complex numbers, and thus $L_0(v)\not\in \mathbb{Z}v$.

\subsubsection{Local description of the action of $\mathbb{G}_m$ on the sheaf of sections.} 
\label{descrMc} 
As above, locally on $CJ$, the sheaf $\mathscr{M}_{CJ}$ is $M \otimes R[w, w^{-1}]$.
The action of an element $z \in \mathbb{C}^\times$ on $\mathscr{M}_{CJ}\otimes  \omega_C^{\otimes -a}$ is then:
\[ M \otimes R[w, w^{-1}] \to  M \otimes R[w, w^{-1}], \quad v\otimes  w^{n+a} \mapsto z^{{}-\deg v +n}\,  v \otimes w^{n+a}.
\]
This description implies that the invariants in $\mathscr{M}_{CJ}\otimes  \omega_C^{\otimes -a}$ under this action are linear combinations of elements of  type $v \otimes w^{\deg v +a}$ for $v \in M$, hence sections of $\mathscr{M}_C$ are spanned by $v \otimes w^{\deg v }$ for $v \in M$. We will use this description in \S \ref{ProjConnVM}.

\subsubsection{The case $M=V$}
\label{VC}
When $M=V$, the action of $L_0=-z\partial_z$ defines the integral gradation of $V$, that is, $v\in V_i$ if and only if $L_0(v)=iv$. It follows that the action of $L_0$ gives rise to a $\mathbb{C}^\times$-action on $V$ which coincides with the one described in \eqref{CactM}. Hence the action of $\textrm{Der}_0\, \mathcal{O}=\textrm{Lie}(\textrm{Aut}\,\mathcal{O})$ on $V$ from \S \ref{UV} can be integrated to an action of $\textrm{Aut}\,\mathcal{O}= \mathbb{G}_m \ltimes \textrm{Aut}_+\mathcal{O}$ on $V$. Replacing $\textrm{Aut}_+\mathcal{O}$ with $\textrm{Aut}\,\mathcal{O}$ in \S \ref{Vmodbdle} and \S\ref{conn}, one produces the \textit{vertex algebra bundle} 
\[
{V}_C := \mathscr{A}ut_C \mathop{\times}_{\textrm{Aut}\,\mathcal{O}}  V
\]
on $C$, whose sheaf of sections
\[
\mathscr{V}_C :=  \left( V \otimes q_*p_* \mathscr{O}_{\mathscr{A}ut_C} \right)^{\textrm{Aut}\,\mathcal{O}}
\]
is a quasi-coherent locally free $\mathscr{O}_C$-module.
Here $p\colon \mathscr{A}ut_C \rightarrow CJ$ and $q\colon CJ\rightarrow C$ are the natural projections.

\subsubsection{The connection on $\mathscr{V}_{C}$}
\label{vab}

With an argument analogous to the proof of Proposition \ref{propMCJconnection}, one can show that
$\mathscr{V}_{C}$ has a flat logarithmic connection:
\begin{proposition} \label{propVconnection} 
The sheaf $\mathscr{V}_{C}$ is naturally equipped with a flat logarithmic connection, i.e., there is a  morphism
\begin{equation}
\label{nablaDiff} 
\nabla \colon \mathscr{V}_C \rightarrow \mathscr{V}_C \otimes \omega_C
\end{equation}
arising from an action of the Lie algebra $\mathscr{T}_{C}$ on $\mathscr{V}_{C}$.
\end{proposition}

\begin{proof}
Being $(\textrm{Aut}\,\mathcal{O})$-equivariant, the connection from \eqref{fllogconnAM} with $M=V$ gives rise to a flat logarithmic connection 
\begin{eqnarray}
\label{nablaV}
\mathscr{T}_{CJ}(-\log D) \otimes \mathscr{V}_C \rightarrow \mathscr{V}_C.
\end{eqnarray}
Equivalently, the action of the tangent sheaf $\mathscr{T}_{CJ}(-\log D)$ on $\mathscr{V}_{CJ}$ is $\mathbb{C}^\times$-equiva\-riant, so we obtain that $\mathscr{T}_{CJ}(-\log(D))^{\mathbb{C}^\times}$ acts on $\mathscr{V}_C$. The action of $L_0$ on $V$ induces the grading on $V$, and thus induces the $\mathbb{C}^\times$ from \S\ref{DesM}. Using an argument similar to the one given for the proof of Proposition \ref{propMCJconnection}, it follows that $\mathscr{T}_{C}(-\log D)$, which coincides with $\mathscr{T}_C$, acts on $\mathscr{V}_C$. For what follows it is more convenient to describe this action on $\mathscr{V}_C$ in terms of differentials: we can rewrite the natural action of $\mathscr{T}_C$ on $\mathscr{V}_C$ as a morphism
\[\Omega_{C}^\vee \otimes \mathscr{V}_C \rightarrow \mathscr{V}_C
\] 
that via the canonical map $\Omega_C \to \omega_C$, induces a map $\omega_C^\vee \otimes \mathscr{V}_C \rightarrow \mathscr{V}_C$, equivalent to  $\nabla$ in \eqref{nablaDiff}.
\end{proof}


\section{Spaces of coinvariants}
\label{Coinvariants}

For a stable $n$-pointed curve $(C,P_{\bullet})$ and $V$-modules $M^\bullet=(M^1, \dots, M^n)$,
we describe  vector spaces of coinvariants ${\mathbb{V}^J}(V; M^\bullet)_{(C,P_\bullet, \tau_\bullet)}$  at  $(C,P_\bullet, \tau_\bullet)$ (dependent on a choice of non-zero $1$-jets $\tau_\bullet$ at the marked points, see (\ref{spacecon})), and ${\mathbb{V}}(V; M^\bullet)_{(C,P_\bullet)}$  at  $(C,P_\bullet)$ (independent of jets, see (\ref{spaceconCstar})) given by the action of the Lie algebra $\mathscr{L}_{C\setminus P_\bullet}(V)$ associated to $(C,P_{\bullet})$ and the vertex operator algebra~$V$ (see (\ref{UCminusPV})).   
Recall that for a representation $\mathbb{M}$ of $\mathscr{L}_{C\setminus P_\bullet}(V)$, the \textit{space of coinvariants} of $\mathbb{M}$  is the quotient
$\mathbb{M} /  \mathscr{L}_{C\setminus P_\bullet}(V)\cdot \mathbb{M}$; this is the largest quotient of $\mathbb{M}$ on which $\mathscr{L}_{C\setminus P_\bullet}(V)$ acts trivially.  

As explained in \S \ref{UPV}, for each marked point $P_i$, the Lie algebra $\mathscr{L}_{P_i}(V)$ (see (\ref{UCminusP})) gives a coordinate-free realization of 
 $\mathfrak{L}(V)$ (see (\ref{LV})) based at $P_i$ by \cite[Cor.~19.4.14]{bzf}.
Following Lemma~\ref{Lem:VonMcoordFree}, the Lie algebra $\mathscr{L}_{P_i}(V)$ acts on the coordinate-free realization of the $V$-module $M^i$ based at $P_i$.
In \S \ref{coinvdef} we will see that this induces an action of $\mathscr{L}_{C\setminus P_\bullet}(V)$ on the coordinate-free realization of $\otimes_{i} M^i$ based at the marked points $P_\bullet$.  

Finally, in \S \ref{VirLater}, we gather a few statements concerning the Virasoro vertex algebra $\textrm{Vir}_c$ and the action of $\mathscr{T}_C(C\setminus P_\bullet)$ ---  the Lie algebra  of regular vector fields on $C\setminus P_\bullet$ --- on spaces of coinvariants that will be used  in \S \ref{projconnCJ} and \S \ref{ProjConnVM}.  The residue theorem plays a role  in the proof of Lemma \ref{vect}.

\subsection{The Lie algebra $\mathscr{L}_{C\setminus P_\bullet}(V)$}
\label{PFC}

Given a stable $n$-pointed curve $(C,P_\bullet)$ and a vertex operator algebra $V$, the vector space
\begin{align}
\label{UCminusPV}
\mathscr{L}_{C\setminus P_\bullet}(V):=H^0 \left( C\setminus P_\bullet, \mathscr{V}_C \otimes \omega_C / \textrm{Im} \nabla \right)
\end{align}
is a Lie algebra (as in the case $C$ smooth proved in \cite[Cor.~19.4.14]{bzf}).
Here $\mathscr{V}_C$ is the vertex algebra sheaf on $C$ from \S\ref{VC}, and $\nabla$ its flat connection from \eqref{nablaDiff}.

\subsection{The Lie algebra $\mathscr{L}_{P}(V)$}
\label{UPV} 
For a stable pointed curve $(C,P)$, let $D^\times_P$ be the punctured disk about  $P$ on $C$.  That is,  $D^\times_P= \textrm{Spec} \, \mathscr{K}_P$, where $\mathscr{K}_P$ is the field of fractions of ~$\widehat{\mathscr{O}}_P$, the completed local ring at $P$. Given a vertex operator algebra $V$, let $\mathscr{V}_C$ be the vertex algebra sheaf on $C$ from \S\ref{VC}, and $\nabla$ its flat connection from \eqref{nablaDiff}.  Define
\begin{align}
\label{UCminusP}
\mathscr{L}_{P}(V) := H^0 \left( D^\times_P, \mathscr{V}_C \otimes \omega_C / \textrm{Im} \nabla \right).
\end{align}
The vector space $\mathscr{L}_{P}(V)$ forms a Lie algebra isomorphic to $\mathfrak{L}(V)$ \cite[Cor.~19.4.14]{bzf}, and can be thought of as a coordinate-independent realization of $\mathfrak{L}(V)$ based at $P$. Moreover, the action of $\mathfrak{L}(V)$ on a $V$-module $M$ extends to an action of the Lie algebra $\mathscr{L}_{P}(V)$  on ${M}_{P,\tau}$ (see \eqref{MPtau}), for any non-zero $1$-jet $\tau$ at~$P$ \cite[Thm 7.3.10]{bzf}. 
Similarly, the following lemma shows that the action of $\mathfrak{L}(V)$ on M induces a coordinate-independent action of  $\mathscr{L}_{P}(V)$ on $M_P$ ($M_P$ is defined in \S\ref{DesM}). 

\begin{lemma} 
\label{Lem:VonMcoordFree} 
There exists an anti-homomorphism of Lie algebras
\[
\alpha_M\colon \mathscr{L}_{P}(V) \rightarrow \textup{End}(M_P).
\]
\end{lemma}
When the action of $\mathbb{G}_m$ is induced by $L_0$, Lemma \ref{Lem:VonMcoordFree} is the content of \cite[\S 7.3.7]{bzf}.  
In the general case where the action of $\mathbb{G}_m$ is given by the integral gradation of $M$, the proof proceeds as in \cite[\S 6.5.4]{bzf}. We present it here for completeness. 
\proof 
One first shows that it is possible to define a coordinate-free version of the map $Y^M(-,z)$. This is a map $\mathcal{Y}^M_P$ which assigns to every section in $H^0( D_P^\times, \mathscr{V}_C)$ an element of $\text{End}(\mathscr{M}_{D_P^\times})$. Let us fix elements $s \in H^0( D_P^\times, \mathscr{V}_C)$, $v \in M_P$, and $\phi \in M_P^\vee$. To define $\mathcal{Y}^M_P$, it is enough to assign an element 
\[ 
\langle \phi, \mathcal{Y}^M_P(s) v \rangle \in \mathscr{O}_C(D_P^\times)
\] 
to every $s$, $v$, and $\phi$ as above.
Once we choose a local coordinate $z$ at $P$, we can identify $s$ with an element $A_z$ of $V\llbracket z \rrbracket$, the element $v$ with an element $v_z$ of $M$ and $\phi$ with an element $\phi_z$ of $M^\vee$. Since the map $\mathcal{Y}^M_P$ must be $\mathscr{O}_{D_P}$-linear, we can further assume that $A_z \in V$, and that it is homogeneous. We then define 
\[ 
\langle \phi, \mathcal{Y}^M_P(s) v \rangle := \phi_z \left( Y^M(A_z,z) v_z \right).
\]
We have to show that this is independent of the choice of the coordinate~$z$. It has already been shown in \cite[\S 7.3.10]{bzf} that this map is $(\text{Aut}_+\mathcal{O})$-invariant, so we only show that it is invariant under the change of coordinate $z \to w:=a z$ for $a \in \mathbb{C}^\times$, i.e., we are left to prove the equality
\begin{equation} 
\label{proofVonMcoordFree} 
\phi_z \left(Y^M(A_z,z) v_z \right)  = \phi_w \left( Y^M(A_w,w) v_w\right). 
\end{equation} 
We need to compare $A_z$, $v_z$, and $\phi_z$ with $A_w$, $v_w$, and $\phi_w$. Unraveling the action of $\mathbb{C}^\times$ on $V$ and $M$ coming from the change of coordinate as described in \eqref{CactM}, we see that 
\[
A_w = a^{-1} \cdot A_z = a^{\deg A_z}A_z, \qquad v_w= a^{-1} \cdot v_z = a^{\deg v_z} v_z,
\] 
where we assume that $v_z$ is homogeneous for simplicity. Finally, the element $\phi_w$ is just the composition of $\phi_z$ with the action of $a$ on $M$. We can then rewrite the right-hand side of equation \eqref{proofVonMcoordFree} as
\[ 
 (\phi_z \circ a ) \left( Y^M(a^{\deg A_z} A_z, az) a^{\deg v} v_z \right)
\]
so it is enough to show that 
\[ 
Y^M(A_z,z)v_z = a \cdot \left( Y^M(a^{\deg A_z} A_z, az) a^{\deg v} v_z\right) \quad\mbox{in $M \llbracket z, z^{-1} \rrbracket$.}
\]
One has
\begin{multline*}
       a \cdot \left( (a^{\deg A_z} A_z)_{(i)} a^{\deg v} v (az)^{-i-1} \right)
    =   a^{\deg A_z + \deg v_z-i-1}  a \cdot\left((A_z)_{(i)}v_z \right) z^{-i-1}\\
    =   a^{\deg A_z + \deg v_z-i-1-\deg (A_z)_{(i)}(v)} (A_z)_{(i)}v_z z^{-i-1}
    = (A_z)_{(i)}v_z z^{-i-1}.
\end{multline*}
The last equality follows from \eqref{ModShift}. We conclude that $\mathcal{Y}^M_P$ is well defined. 
In fact this is a particular case of a more general result proved in \cite{yzhuang}. 
Using residues and an argument similar to \cite[\S 6.5.8]{bzf}, we induce from $\mathcal{Y}^M_P$ the dual map
\[
\left(\mathcal{Y}^{M}_P\right)^\vee\colon H^0(D^\times_P, \mathscr{V} \otimes \omega_C) \to \text{End}(M_P).
\]
It remains to show that this factors through a map 
\[
\alpha_M \colon H^0(D^\times_P, \mathscr{V} \otimes \omega_C / \textrm{Im} \nabla) \to \text{End}(M_P).
\]
This can be checked by choosing a local coordinate at $P$ as in \cite[\S 7.3.7]{bzf}.
\endproof

\subsection{Coinvariants}
\label{coinvdef}
Fix a stable $n$-pointed curve $(C,P_\bullet)$ with marked points $P_\bullet=(P_1,\dots, P_n)$. 
One has a Lie algebra homomorphism by restricting sections
\begin{equation}
\label{LCPtosumLP}
\mathscr{L}_{C\setminus P_\bullet}(V)  \rightarrow  \oplus_{i=1}^n \mathscr{L}_{P_i}(V), \ \ \
\mu \mapsto  \left(\mu_{P_1},\dots, \mu_{P_n} \right).
\end{equation}
This is as in the case $C$ smooth proved in \cite[Cor.~19.4.14]{bzf}.

We assume below that each irreducible component of $C$ contains a marked point; in particular, $C\setminus P_\bullet$ is affine.

Fix $\tau_\bullet=(\tau_1,\dots, \tau_n)$, where each $\tau_i$ is a non-zero $1$-jet at $P_i$. 
Given an $n$-tuple of $V$-modules $M^\bullet:=(M^1,\dots,M^n)$, the \textit{space of coinvariants} of $M^{\bullet}$ at  $(C,P_\bullet, \tau_\bullet)$ is defined as the quotient
\begin{align}
\label{spacecon}
{\mathbb{V}^J}(V; M^\bullet)_{(C,P_\bullet, \tau_\bullet)} := \otimes_{i=1}^n \, {M}_{P_i, \tau_i}^i \Big/  \mathscr{L}_{C\setminus P_\bullet}(V)\cdot \left( \otimes_{i=1}^n \, {M}_{P_i, \tau_i}^i \right).
\end{align}
Here, ${M}_{P_i, \tau_i}^i$ is the coordinate-independent realization of $M^i$ assigned at $(P_i,\tau_i)$ as in \eqref{MPtau}.
The action of $ \mathscr{L}_{C\setminus P_\bullet}(V)$ is the restriction of the action of $\oplus_{i=1}^n \mathscr{L}_{P_i}(V)$ via \eqref{LCPtosumLP}:
\[
\mu  (A_1\otimes \cdots \otimes A_n) = \sum_{i=1}^n A_1\otimes \cdots \otimes \alpha_{M^i}(\mu_{P_i})\cdot A_ i \otimes \cdots \otimes A_n,
\]
for $A_i \in {M}_{P_i, \tau_i}^i$ and $\mu\in  \mathscr{L}_{C\setminus P_\bullet}(V)$.

Similarly, the \textit{space of coinvariants} of $M^{\bullet}$ at  $(C,P_\bullet)$ is defined as the quotient
\begin{align}
\label{spaceconCstar}
{\mathbb{V}}(V; M^\bullet)_{(C,P_\bullet)}:=\otimes_{i=1}^n \, {M}_{P_i}^i \Big/  \mathscr{L}_{C\setminus P_\bullet}(V)\cdot \left( \otimes_{i=1}^n \, {M}_{P_i}^i \right).
\end{align}
Here, ${M}_{P_i}^i$ is the coordinate-independent realization of $M^i$ assigned at $P_i$ as in \S\ref{DesM}.

\subsection{Triviality of the action of $\mathscr{T}_C(C\setminus P_\bullet)$ on coinvariants}\label{VirLater}
For $c \in \mathbb{C}$, let $U(\textrm{Vir})$ be the universal enveloping algebra of the Virasoro algebra $\textrm{Vir}$, and define
\[
\textrm{Vir}_c={\rm Ind}^{\textrm{Vir}}_{\textrm{Der}\, \mathcal{O} \oplus \mathbb{C}K} \,\mathbb{C}_c =  U(\textrm{Vir}) \otimes_{U(\textrm{Der}\, \mathcal{O} \oplus \mathbb{C}K)} \mathbb{C}_c,
\]
where  $\textrm{Der}\, \mathcal{O}$ acts by zero and $\textrm{K}$ acts as multiplication by $c$  on the one dimensional module $\mathbb{C}_c\cong \mathbb{C}$.
The space $\textrm{Vir}_c$ has the structure of a vertex operator algebra of central charge $c$ (see \cite[\S 2.5.6]{bzf} for this and more details about $\textrm{Vir}_c$).  

Let $\mathscr{T}_C(C\setminus P_\bullet)$ be the Lie algebra  
of regular vector fields on $C\setminus P_\bullet$.  For  $V$-modules $M^\bullet=(M^1, \dots, M^n)$,  we will see that the action of $\mathscr{L}_{C\setminus P_\bullet}(\textup{Vir}_c)$ extends an action of $\mathscr{T}_C(C\setminus P_\bullet)$ on coordinate-free realizations of $\otimes_i M^i$, hence:

\begin{lemma}\label{vect}
When $C\setminus P_\bullet$ is affine,
 $\mathscr{T}_C(C\setminus P_\bullet)$ acts trivially on  
${\mathbb{V}^J}(\textup{Vir}_c; M^\bullet)_{(C,P_\bullet, \tau_\bullet)}$  and on  $\mathbb{V}(\textup{Vir}_c; M^\bullet)_{(C,P_\bullet)}$.
\end{lemma}

\begin{proof}
When $n=1$,  one has an inclusion $\mathscr{T}_C(C\setminus P) \hookrightarrow \mathscr{L}_{P}(\textrm{Vir}_c)$ by \cite[\S 19.6.5]{bzf}.
Indeed, $\mathscr{L}_{C\setminus P}(\textrm{Vir}_c)$ contains as a Lie subalgebra an extension of $\mathscr{T}_C(C\setminus P)$ by $H^0(C\setminus P, \omega/d\mathscr{O})$, and by the residue theorem, the image of  $H^0(C\setminus P, \omega/d\mathscr{O})$ via the restriction $\mathscr{L}_{C\setminus P}(\textrm{Vir}_c)\rightarrow \mathscr{L}_{P}(\textrm{Vir}_c)$ is zero.
It follows that
\[
\mathscr{T}_C(C\setminus P) \hookrightarrow \textrm{Im}\left( \mathscr{L}_{C\setminus P}(\textrm{Vir}_c) \rightarrow \mathscr{L}_{P}(\textrm{Vir}_c) \right).
\]

For more marked points, as in \cite[\S 19.6.5]{bzf}, one has that 
\[
\textrm{Im}\left( \mathscr{L}_{C\setminus P_\bullet}(\textrm{Vir}_c) \rightarrow \oplus_{i=1}^n \mathscr{L}_{P_i}(\textrm{Vir}_c) \right)
\]
contains an extension of $\mathscr{T}_C(C\setminus P_\bullet)$ by the image of the map $\varphi$
\[
\begin{tikzcd}
H^0(C\setminus P_\bullet, \omega_C/d\mathscr{O}_C) \arrow[rightarrow]{r}{\varphi}\arrow[hookrightarrow]{d} & \oplus_{i=1}^n H^0(D^\times_{P_i}, \omega_C/d\mathscr{O}_C)\arrow[Isom]{d}  \arrow[hookrightarrow]{r} & \oplus_{i=1}^n \mathscr{L}_{P_i}(\textrm{Vir}_c)\\
 \mathscr{L}_{C\setminus P_\bullet}(\textrm{Vir}_c) &  \mathbb{C}^n &
\end{tikzcd}
\]
inside the central space (isomorphic to) $\mathbb{C}^n$ of $\oplus_{i=1}^n \mathscr{L}_{P_i}(\textrm{Vir}_c)$.
The map $\varphi$ is obtained by restricting sections. The space of sections of $\omega_C/d\mathscr{O}_C$ on $D^\times_{P_i}$ is isomorphic to $\mathbb{C}$ for each $i$,  as such sections  are identified by their residue at~$P_i$.
The inclusions above are induced from  $\omega_C\hookrightarrow \mathscr{V}ir_c\otimes \omega_C$, which gives an inclusion
$\omega_C/d\mathscr{O}_C\hookrightarrow  \mathscr{V}ir_c\otimes \omega_C/\textrm{Im}\nabla$. Here, $\mathscr{V}ir_c$ is the sheaf $\mathscr{V}_C$ for $V=\textrm{Vir}_c$.

As $C\setminus P_\bullet$ is affine, one has $H^0(C\setminus P_\bullet,  \omega_C)\twoheadrightarrow H^0(C\setminus P_\bullet, \omega_C/d\mathscr{O}_C)$.
By the residue theorem, the image of $\varphi$ consists of the hyperplane of points $(r_1,\dots,r_n)$ in $\mathbb{C}^n$ such that $r_1+\cdots+r_n=0$.
Following the identification of the Virasoro algebra in \eqref{eq:VirinLV}, each central space $\mathbb{C}\subset  \mathscr{L}_{P_i}(\textrm{Vir}_c)$ is identified with $\mathbb{C}|0\rangle_{[-1]}$. Thus, since $|0\rangle_{[-1]}$ acts as the identity, the image of $\varphi$ acts trivially.
\end{proof}

For any vertex operator algebra $V$ of central charge $c$, there is a vertex algebra homomorphism $\textrm{Vir}_c\rightarrow V$ \cite[\S 3.4.5]{bzf}.  This gives a surjection
$\mathbb{V}(\textup{Vir}_c; M^\bullet)_{(C,P_\bullet)} \twoheadrightarrow \mathbb{V}(V; M^\bullet)_{(C,P_\bullet)}$
and similarly for the coinvariants $\mathbb{V}^J(-; M^\bullet)_{(C,P_\bullet, \tau_\bullet)}$
\cite[\S 10.2.2]{bzf}.

\medskip

The following statement will be used in the proof of Theorem \ref{AtalgonCJ}.

\begin{lemma}
\label{UCPVircinUCPVLemma}
 $\mathscr{T}_C(C\setminus P_\bullet)$ 
acts trivially on   ${\mathbb{V}^J}(V; M^\bullet)_{(C,P_\bullet, \tau_\bullet)}$ and $\mathbb{V}(V; M^\bullet)_{(C,P_\bullet)}$.
\end{lemma}

\begin{proof}
Equivalent to the surjection on coinvariants referred to above, one has inclusions
\begin{align*}
\mathscr{L}_{C\setminus P_\bullet}(\textrm{Vir}_c) \cdot \left( \otimes_{i=1}^n M^i_{P_i} \right) \hookrightarrow \mathscr{L}_{C\setminus P_\bullet}(V)  \cdot \left( \otimes_{i=1}^n M^i_{P_i} \right),\\
\mathscr{L}_{C\setminus P_\bullet}(\textrm{Vir}_c) \cdot \left( \otimes_{i=1}^n M^i_{P_i,\tau_i} \right) \hookrightarrow \mathscr{L}_{C\setminus P_\bullet}(V)  \cdot \left( \otimes_{i=1}^n M^i_{P_i,\tau_i} \right).
\end{align*}
Therefore, Lemma \ref{UCPVircinUCPVLemma} follows from Lemma \ref{vect}.
\end{proof}


\section{Sheaves of coinvariants}
\label{SheavesOfCoinvariants}
In this section we define the sheaves of coinvariants $\widetriangle{\mathbb{V}}_g(V; M^\bullet)$, $\mathbb{V}^J_g(V; M^\bullet)$, and ${\mathbb{V}}_g(V; M^\bullet)$ on the moduli spaces 
$\widetriangle{\mathcal{M}}_{g,n}$,  $\overline{\mathcal{J}}_{g,n}^{1,\times}$, and $\overline{\mathcal{M}}_{g,n}$, respectively.  We will sometimes drop the $g$ or the $V$ in the notation (or both), when the context makes it clear.

Working in families, one defines a sheaf of Lie subalgebras $\mathcal{L}_{\mathcal{C}\setminus\mathcal{P}_\bullet}(V)$ with fibers equal to the Lie algebra $\mathscr{L}_{C\setminus P_\bullet}(V)$ from \eqref{UCminusPV}. This is used  to define the sheaf of coinvariants  
$\widetriangle{\mathbb{V}}_g(V; M^\bullet)$ on $\widetriangle{\mathcal{M}}_{g,n}$ (\S \ref{SheafOfCoinvariantsTriangle}). In \S \ref{FirstDescent} we will see that $\widetriangle{\mathbb{V}}_g(V; M^\bullet)$ is an $(\textrm{Aut}_+\mathcal{O})^n$-equivariant $\mathcal{O}_{\widetriangle{\mathcal{M}}_{g,n}}$-module, hence descends along the principal $(\textrm{Aut}_+\mathcal{O})^n$-bundle $\widetriangle{\pi}\colon\widetriangle{\mathcal{M}}_{g,n}\rightarrow \overline{\mathcal{J}}_{g,n}^{1,\times}$, allowing one to describe the sheaf $\mathbb{V}^J_g(V; M^\bullet)$ on $\overline{\mathcal{J}}_{g,n}^{1,\times}$ (Definition~\ref{SheafOfCoinvariants}). A second descent in \S \ref{SecondDescent} gives  the sheaf ${\mathbb{V}}_g(V; M^\bullet)$ on $\overline{\mathcal{M}}_{g,n}$ (Definition~\ref{SOCM}).   
The structure and relationships between the three sheaves of coinvariants are summarized in Figure \ref{fig:Sheavesofcoinvariants}.

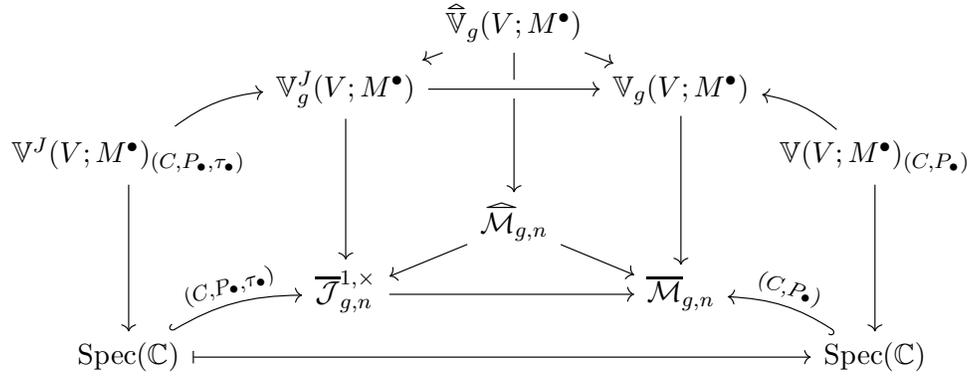
\begin{figure}[htb]
\centering
\begin{tikzcd}[back line/.style={draw}, row sep=0.3em, column sep=0.2em]
{}&& \widetriangle{\mathbb{V}}_g(V; M^\bullet) \ar[rightarrow]{dl} \ar[rightarrow]{dr} \ar[rightarrow, back line]{ddd} &&\\
{}&\mathbb{V}^J_g(V; M^\bullet) \ar[rightarrow, crossing over]{rr} \ar[rightarrow]{ddd}{} &&\mathbb{V}_g(V; M^\bullet) \ar[rightarrow]{ddd}{} &\\
{\mathbb{V}^J}(V; M^\bullet)_{(C,P_\bullet,\tau_\bullet)}\ar[rightarrow]{ddd} \ar[rightarrow, bend left=15]{ur}  &&&&  \mathbb{V}(V; M^\bullet)_{(C,P_\bullet)}\ar[rightarrow]{ddd} \ar[rightarrow, bend left=-15]{ul} \\
&& \widetriangle{\mathcal{M}}_{g,n} \ar[rightarrow]{dl} \ar[rightarrow]{dr} &&\\
{}& \overline{\mathcal{J}}_{g,n}^{1,\times} \ar[rightarrow]{rr}{} && \overline{\mathcal{M}}_{g,n} &\\
\textrm{Spec}(\mathbb{C}) \ar[hookrightarrow, bend left=15, sloped, pos=0.5]{ur}{(C,P_\bullet,\tau_\bullet)} \ar[|->]{rrrr} &&&&\textrm{Spec}(\mathbb{C}) \ar[hookrightarrow, bend left=-15, sloped, pos=0.5]{ul}{(C,P_\bullet)} 
\end{tikzcd}
   \caption{The structure of sheaves of coinvariants.}
   \label{fig:Sheavesofcoinvariants}
\end{figure}

\subsection{Sheaves of Lie algebras on $\overline{\mathcal{M}}_{g,n}$ and on $\widetriangle{\mathcal{M}}_{g,n}$}\label{Lsheaf}
 We will first define a sheaf of Lie algebras $\mathcal{L}_{\mathcal{C}\setminus\mathcal{P}_\bullet}(V)$ on $\overline{\mathcal{M}}_{g,n}$, which we can regard as a sheaf on $\widetriangle{\mathcal{M}}_{g,n}$ by choosing formal coordinates at the marked points (that is, by pulling $\mathcal{L}_{\mathcal{C}\setminus\mathcal{P}_\bullet}(V)$  back to $\widetriangle{\mathcal{M}}_{g,n}$ along the projection $\widetriangle{\mathcal{M}}_{g,n} \rightarrow \overline{\mathcal{M}}_{g,n}$.

  Identify $\pi_{n+1}\colon \overline{\mathcal{M}}_{g,n+1}\rightarrow  \overline{\mathcal{M}}_{g,n}$ with the universal curve  $\pi_{n+1}\colon\overline{\mathcal{C}}_{g,n}\rightarrow\overline{\mathcal{M}}_{g,n}$. Let $\overline{\mathcal{E}}_{g,n}$ be the moduli space of stable $(n+1)$-pointed genus $g$ curves, together with a formal coordinate at the last marked point. The natural projection 
$\overline{\mathcal{E}}_{g,n}\rightarrow \overline{\mathcal{M}}_{g,n+1}$ realizes $\overline{\mathcal{E}}_{g,n}$ as a principal $(\textrm{Aut}\,\mathcal{O})$-bundle on~$\overline{\mathcal{C}}_{g,n}$.

The group $\textrm{Aut}\,\mathcal{O}$ acts equivariantly on the trivial vector bundle $\overline{\mathcal{E}}_{g,n} \times V$ over $\overline{\mathcal{E}}_{g,n}$. The quotient of $\overline{\mathcal{E}}_{g,n} \times V$ modulo the action of $\textrm{Aut}\,\mathcal{O}$ descends to a vector bundle
\[
{V}_{\overline{\mathcal{C}}_{g,n}} := \overline{\mathcal{E}}_{g,n} \mathop{\times}_{\textrm{Aut}\,\mathcal{O}} V
\] 
on $\overline{\mathcal{C}}_{g,n}$. Let $\mathcal{V}_{\overline{\mathcal{C}}_{g,n}}$ denote the sheaf of sections of ${V}_{\overline{\mathcal{C}}_{g,n}}$.
As in \S\ref{VC}, the trivial bundle $\overline{\mathcal{E}}_{g,n} \times V$ has a flat logarithmic connection along the fibers of   $\pi_{n+1}\colon\overline{\mathcal{C}}_{g,n}\rightarrow\overline{\mathcal{M}}_{g,n}$. The connection is $(\textrm{Aut}\,\mathcal{O})$-equivariant, hence descends to a flat logarithmic connection on $\mathcal{V}_{\overline{\mathcal{C}}_{g,n}}$ along the fibers of   $\pi_{n+1}$:
\[
\nabla \colon \mathcal{V}_{\overline{\mathcal{C}}_{g,n}} \rightarrow \mathcal{V}_{\overline{\mathcal{C}}_{g,n}} \otimes \omega_{\pi_{n+1}}.
\]
Define
\begin{equation}\label{LSheaf}
\mathcal{L}_{\mathcal{C}\setminus\mathcal{P}_\bullet}(V):= \left(\pi_{n+1}\right)_* \textrm{coker}\left( \nabla|_{\overline{\mathcal{C}}_{g,n}\setminus\mathcal{P}_\bullet} \right).
\end{equation}
Here $\mathcal{P}_\bullet$ denotes the union of the images of the $n$ sections $\overline{\mathcal{M}}_{g,n}\rightarrow\overline{\mathcal{C}}_{g,n}$ given by the marked points. From \cite[Cor.~19.4.14]{bzf}, $\mathcal{L}_{\mathcal{C}\setminus\mathcal{P}_\bullet}(V)$ is a sheaf of Lie algebras on $\overline{\mathcal{M}}_{g,n}$. Pulling back to $\widetriangle{\mathcal{M}}_{g,n}$, after a choice of coordinates, gives a sheaf that we still call $\mathcal{L}_{\mathcal{C}\setminus\mathcal{P}_\bullet}(V)$ on
$\widetriangle{\mathcal{M}}_{g,n}$.

Next, we consider the sheaf of Lie algebras
\begin{equation}\label{LVSheaf}
{\mathcal{L}}(V)^n := \mathfrak{L}(V)^{\oplus n} \,\,\widehat\otimes\, \mathcal{O}_{\widetriangle{\mathcal{M}}_{g,n}}
\end{equation}
on $\widetriangle{\mathcal{M}}_{g,n}$.
By restricting sections, one can define a homomorphism of sheaves of Lie algebras: 
\[
\varphi\colon \mathcal{L}_{\mathcal{C}\setminus\mathcal{P}_\bullet}(V)\rightarrow {\mathcal{L}}(V)^n. 
\]

\subsection{The sheaf  $\widetriangle{\mathbb{V}}_g(V; M^\bullet)$ on $\widetriangle{\mathcal{M}}_{g,n}$}
\label{SheafOfCoinvariantsTriangle}
Let $\left(\mathcal{C}\rightarrow S, P_\bullet, t_\bullet\right)$ be a family of coordinatized stable pointed curves over a smooth base $S$, with $n$ sections $P_i\colon S\rightarrow \mathcal{C}$ and formal coordinates $t_i$. 
Assume that $\mathcal{C}\setminus P_\bullet(S)$ is affine over $S$ (e.g., $S=\widehat{\mathcal{M}}_{g,n}$). Let $V$ be a vertex operator algebra and let $M^\bullet=(M^1,\dots, M^n)$ be  $V$-modules. 
With notation as above, we set

\begin{definition}
\label{Vtriangle}
The \textit{sheaf of coinvariants} $\widetriangle{\mathbb{V}}(V; M^\bullet)_{(\mathcal{C}\rightarrow S,P_\bullet, t_\bullet)}$ on $S$ is the quasi-coherent sheaf  
\begin{equation*}
\widetriangle{\mathbb{V}}(V; M^\bullet)_{(\mathcal{C}\rightarrow S,P_\bullet, t_\bullet)} := 
\otimes_{i=1}^n M^i \otimes \mathcal{O}_{S} \Big/ \mathcal{L}_{\mathcal{C}\setminus\mathcal{P}_\bullet}(V) \cdot \left( \otimes_{i=1}^n M^i \otimes  \mathcal{O}_{S} \right).
\end{equation*}
 \end{definition}

We can remove the condition that $\mathcal{C}\setminus P_\bullet(S)$ be affine over $S$ as in the construction of
coinvariants obtained from representations of affine Lie algebras, see e.g., \cite{fakhr}, \cite{loowzw}. First we review the property known as propagation of vacua.

\subsubsection{}
Let $\left(\mathcal{C}\rightarrow S, P_\bullet, t_\bullet\right)$ be a family of coordinatized stable pointed curves  as above. 
Let \mbox{$Q\colon S\rightarrow \mathcal{C}$} be a section such that $Q(S)$ is inside the smooth locus in $\mathcal{C}$ and is disjoint from $P_i(S)$, for each $i$, and let $r$ be a formal coordinate at $Q(S)$. Given $V$-modules $M^i$, for $i=1,\dots,n$, set $M^\bullet=\otimes_{i=1}^n M^i$.

\begin{theorem}[Propagation of vacua {\cite[Thm 3.6]{gc}}]
\label{thm:POV}
Let $V=\oplus_{i\geq 0} V_i$ be a vertex operator algebra  such that $V_0\cong \mathbb{C}$.
Assume $\mathcal{C}\setminus P_\bullet(S)$ is affine over $S$. 
The linear map 
\[
M^\bullet \rightarrow M^\bullet \otimes V, \qquad u\mapsto u\otimes |0\rangle
\]
 induces a canonical $\mathscr{O}_S$-module isomorphism
\[
\varphi\colon \widetriangle{\mathbb{V}}\left(V;M^\bullet\right)_{\left(\mathcal{C}\rightarrow S, P_\bullet, t_\bullet\right)} \xrightarrow{\cong}
\widetriangle{\mathbb{V}}\left(V;M^\bullet\otimes V\right)_{\left(\mathcal{C}\rightarrow S, P_\bullet \sqcup Q, t_\bullet\sqcup r\right)}.
\]
Varying $(Q,r)$, the induced isomorphisms are compatible. Moreover, as $|0\rangle$ is fixed by the action of $\mathrm{Aut}\,\mathcal{O}$, the isomorphism is equivariant with respect to change of coordinates.
\end{theorem}

This result is proven in \cite[Thm 3.6]{gc} in the generality we need. We sketch the proof below for completeness. The proof uses that $V$ admits a basis of PBW-type  (\cite{gn}, see also \cite[\S 3]{AbeBuhlDong}); and for this,  $V$ is not required to be $C_2$-cofinite.
Prior to \cite{gc}, propagation vacua has been shown in case the vertex operator algebra $V$ is quasi-primary generated with $V_0\cong \mathbb{C}$, for either \textit{fixed} smooth curves 
 \cite{ZhuGlobal, an1}, and for \textit{rational} stable pointed curves \cite{NT}.  Although in these papers a different Lie algebra is used in the construction, the resulting vector space of coinvariants is isomorphic to the vector space of coinvariants that we consider, in cases where both Lie algebras are defined \cite[Proposition A.2.1]{DGT2}.
 For  a fixed smooth curve, the analogous  statement for spaces of conformal blocks  has been proved in \cite[\S 10.3.1]{bzf}, generalizing the similar result from \cite{tuy}
for sheaves of coinvariants obtained from representations of affine Lie algebras. 

\begin{proof}[Sketch of proof of Theorem \ref{thm:POV}]
We start by showing that the induced map is well-defined. Assume $\sigma\in \mathscr{L}_{C\setminus P_\bullet}(V)$, $u\in M^\bullet$, and $f\in \mathscr{O}_S$. 
Since $\sigma$ is regular at $Q$, its Laurent series expansion $\sigma_Q$ at $Q$ is a linear combination of $\sum_{i\geq 0} A\otimes r^i$, for  $A\in V$. From the axiom on the vacuum vector, one has $\sum_{i\geq 0} A_{(i)}|0\rangle=0$, hence $\sigma_Q\left(|0\rangle \right)=0$.
It follows that
\[
\varphi\left(\sigma\left( u\otimes f \right)\right) =\varphi\left(\sigma\left( u\right) \otimes f \right) 
=  \sigma( u)\otimes |0\rangle \otimes f  + u\otimes \sigma_Q \left( |0\rangle \right)\otimes f= \sigma \left( u\otimes |0\rangle \otimes f\right).
\]
Since $\mathscr{L}_{C\setminus P_\bullet}(V)\subset \mathscr{L}_{C\setminus (P_\bullet\sqcup Q)}(V)$, this shows that $\varphi$ maps  zero  to zero, hence it is well defined.

\smallskip

To show that $\varphi$ is surjective: Assume $u\otimes B \otimes f\in M^\bullet \otimes V \otimes \mathscr{O}_S$.
From the axiom on the vacuum, one has $B=B_{(-1)}|0\rangle$. By an application of the Riemann-Roch theorem, there exists a section $\mu$ of the relative dualizing sheaf $\omega_{\mathcal{C}/S}$ having at most a simple pole at $Q$  and regular on $C\setminus (P_\bullet \sqcup Q)$ --- such a section can be found by requiring poles of high order at the points~$P_\bullet$. To simplify the notation, rescale $\mu$ so that it has residue $1$ at $Q$. 
Find a section $\mathcal{B}$ of $\mathcal{V}_{\mathcal{C}}\rightarrow \mathcal{C}$ such that its restriction to $Q$ is $\mathcal{B}_Q=B$.
By the axiom on the vacuum, the Laurent series expansion $B\otimes \mu_Q$ of $\mathcal{B}\otimes \mu$ at $Q$ satisfies
$\left(B\otimes \mu_Q\right)|0\rangle = B_{(-1)} |0\rangle=B$. Hence one has
\begin{align*}
u\otimes B \otimes f &= u\otimes \left(B\otimes \mu_Q\right)\left(|0\rangle\right) \otimes f\\
&\equiv  - \sum_{i=1}^n u_1 \otimes \cdots \otimes \left( \mathcal{B}_{P_i}\otimes \mu_{P_i} \right)(u_i) \otimes \cdots u_n\otimes |0\rangle \otimes f \,\,\, \mbox{mod } \mathscr{L}_{C\setminus (P_\bullet\sqcup Q)}(V)\\
&= -\varphi\left(\sum_{i=1}^n u_1 \otimes \cdots \otimes\left( \mathcal{B}_{P_i}\otimes \mu_{P_i} \right)(u_i) \otimes \cdots \otimes u_n\otimes f \right)
\end{align*}
where $u=u_1\otimes \cdots \otimes u_n\in M^\bullet$.
This completes the proof of surjectivity.

\smallskip

To show the injectivity of $\varphi$, one shows the surjectivity of the dual map
\[
\varphi^\vee \colon \mathrm{Hom}\left(\widetriangle{\mathbb{V}}\left(V;M^\bullet\otimes V\right)_{\left(\mathcal{C}\rightarrow S, P_\bullet \sqcup Q, t_\bullet\sqcup r\right)} ,\mathscr{O}_S\right)\rightarrow 
\mathrm{Hom}\left(\widetriangle{\mathbb{V}}\left(V;M^\bullet\right)_{\left(\mathcal{C}\rightarrow S, P_\bullet, t_\bullet\right)} ,\mathscr{O}_S\right).
\]
Given a conformal block $\Psi$ in $\mathrm{Hom}\left(\widetriangle{\mathbb{V}}\left(V;M^\bullet\right)_{\left(\mathcal{C}\rightarrow S, P_\bullet, t_\bullet\right)} ,\mathscr{O}_S\right)$, one defines a conformal block $\widetilde{\Psi}$ in $\mathrm{Hom}\left(\widetriangle{\mathbb{V}}\left(V;M^\bullet\otimes V\right)_{\left(\mathcal{C}\rightarrow S, P_\bullet \sqcup Q, t_\bullet\sqcup r\right)} ,\mathscr{O}_S\right)$ so that $\varphi^\vee\left( \widetilde{\Psi}\right)={\Psi}$. 
Start by defining $\widetilde{\Psi}\left(u\otimes |0\rangle\right) = \Psi(u)$, for $u\in M^\bullet$.
In general, assume $u\otimes B\in M^\bullet \otimes V$. 
Since $V_0\cong \mathbb{C}$, the vertex operator algebra $V$ admits a spanning set of PBW-type \cite{gn}. Choose a basis of $V$ of PBW-type, and write
$B=B^k_{(i_k)} \cdots B^1_{(i_i)} |0\rangle$, for some $B^i\in V$ determined by such basis. 
Define $\overline{B}:=A^{k-1}_{(i_{k-1})} \cdots A^1_{(i_i)} |0\rangle$, that is, $B=A^k_{(i_k)} \overline{B}$.  Choose an integer $m$ such that $B^k_{(i)}\overline{B}=0$, for all $i\geq m$.
One may find a section $\mu$ of $\omega_{\mathcal{C}/S}$ such that  $\mu_Q\in r^{i_k}dr+ r^m\mathscr{O}_S\llbracket r\rrbracket dr$ and is regular on $C\setminus (P_\bullet \sqcup Q)$;  and a section $\mathcal{B}$ of $\mathcal{V}_{\mathcal{C}}\rightarrow \mathcal{C}$ such that its restriction to $Q$ is $\mathcal{B}_Q=B^k$.
Define $\widetilde{\Psi}$ so that
\[
\widetilde{\Psi}(u\otimes B) = -\widetilde{\Psi}\left( \sum_{i=1}^n u_1\otimes \cdots \otimes \left(\mathcal{B}_{P_i}\otimes \mu_{P_i}\right)(u_i) \otimes \cdots \otimes u_n \otimes \overline{B} \right),
\]
where $u=u_1\otimes \cdots \otimes u_n\in M^\bullet$. By recursion on $k$, this defines the linear functional $\widetilde{\Psi}$ of $M^\bullet\otimes V$.

One  verifies that $\widetilde{\Psi}$ vanishes on $\mathscr{L}_{\mathcal{C}\setminus (P_\bullet \sqcup Q)}(V)\left( M^\bullet \otimes V\right)$, so that $\widetilde{\Psi}$ is indeed in the source of~$\varphi^\vee$. 
Finally, by definition of $\varphi^\vee$, one has 
$\varphi^\vee\left( \widetilde{\Psi}\right)(u)=\widetilde{\Psi}(u\otimes |0\rangle)$ for $u\in M^\bullet$.
This shows that $\varphi^\vee\left( \widetilde{\Psi}\right)=\Psi$ and completes the proof of the statement.
\end{proof}

\subsubsection{} Consider a family of stable pointed curves $\left(\mathcal{C}\rightarrow S, P_\bullet\right)$ as above, but without assuming that $\mathcal{C}\setminus P_\bullet(S)$ be affine over $S$. After an \'etale base change, we can assume that the family $\left(\mathcal{C}\rightarrow S, P_\bullet\right)$ has $m$ additional sections $Q_i\colon S \rightarrow\mathcal{C}$ such that $\mathcal{C}\setminus (P_\bullet(S)\sqcup Q_\bullet(S))$ is affine over $S$ and $\left(\mathcal{C}\rightarrow S, P_\bullet\sqcup Q_\bullet\right)$ is stable. Let $t_\bullet$ and $r_\bullet$ be the formal coordinates at $P_\bullet(S)$ and $Q_\bullet(S)$. By Definition \ref{Vtriangle} we can define
\[
\widetriangle{\mathbb{V}}(V; M^\bullet\cup(V,\dots,V))_{(\mathcal{C}\rightarrow S,P_\bullet\sqcup Q_\bullet,  t_\bullet\sqcup r_\bullet)}
\] 
and thanks to Theorem \ref{thm:POV} this is independent of the choice of $(Q_\bullet, r_\bullet)$. We then define 
\[
\widetriangle{\mathbb{V}}(V; M^\bullet)_{(\mathcal{C}\rightarrow S,P_\bullet,  t_\bullet)}
\]
as the sheaf associated to the presheaf which assigns to each $U \subset S$ the module
\[
\varinjlim_{(Q_\bullet, r_\bullet)} \widetriangle{\mathbb{V}}(V; M^\bullet\cup(V,\dots,V))_{(\mathcal{C} \rightarrow U ,P_\bullet\sqcup Q_\bullet,  t_\bullet\sqcup r_\bullet)}, 
\] 
where $(Q_\bullet, r_\bullet)$ runs over all sections of $\mathcal{C}_U  \to U$ such that $\mathcal{C}_U \setminus ( P_\bullet(U) \sqcup Q_\bullet(U))$ is affine over $U$ and $\left(\mathcal{C}\rightarrow U, P_\bullet\sqcup Q_\bullet\right)$ is stable. 

As the construction of the sheaf of coinvariants commutes with base change, one obtains a sheaf of coinvariants on $\widetriangle{\mathcal{M}}_{g,n}$, denoted $\widetriangle{\mathbb{V}}_g (V; M^\bullet)$.

\subsection{Two descents along torsors}

\subsubsection{The first descent to define ${\mathbb{V}^J}(V; M^\bullet)$ on $J:=\overline{\mathcal{J}}_{g,n}^{1,\times}$}\label{FirstDescent}
The group $(\textrm{Aut}_+\mathcal{O})^n$ acts by conjugation on $\mathfrak{L}(V)^{\oplus n}$. 
Together with the action of $(\textrm{Aut}_+\mathcal{O})^n$ on $\widetriangle{\mathcal{M}}_{g,n}$, this gives rise to an equivariant action of $(\textrm{Aut}_+\mathcal{O})^n$ on the sheaf ${\mathcal{L}}(V)^n$.
We will need the following statement:

\begin{lemma}
\label{UVnoutAut}
The subsheaf $\varphi\left(\mathcal{L}_{\mathcal{C}\setminus\mathcal{P}_\bullet}(V)\right)$ of ${\mathcal{L}}(V)^n$ is preserved by the action of $(\textup{\textrm{Aut}}_+\mathcal{O})^n$.
\end{lemma}
 
More precisely, let $\mathscr{L}'_{C\setminus P_\bullet}(V)$ be the fiber of $\varphi\left( \mathcal{L}_{\mathcal{C}\setminus\mathcal{P}_\bullet}(V)\right)$ at $(C, P_\bullet, t_\bullet)$.
One has 
\[
\mathscr{L}'_{C\setminus P_\bullet}(V) = \rho^{-1} \mathscr{L}'_{C\setminus P_\bullet}(V) \rho
\]
for all $\rho\in (\textrm{Aut}_+\mathcal{O})^n$. This is a special case of \eqref{urhourho}.

Let $\left(\mathcal{C}\rightarrow S, P_\bullet, \tau_\bullet \right)$ be a family of stable pointed curves over a smooth base $S$ with $n$ sections $P_i\colon S\rightarrow \mathcal{C}$ and non-zero $1$-jets $\tau_i$. 
Assume that $\mathcal{C}\setminus P_\bullet(S)$ is affine over $S$ (e.g., $S={\mathcal{J}}_{g,n}^{1,\times}$, the locus of smooth curves in $\overline{\mathcal{J}}_{g,n}^{1,\times}$).
Define the principal $(\textrm{Aut}_+\mathcal{O})^n$-bundle $\widetriangle{S}\rightarrow S$ as the pull-back of $\widetriangle{\pi}\colon \widetriangle{\mathcal{M}}_{g,n} \rightarrow \overline{\mathcal{J}}_{g,n}^{1,\times}$ via the moduli map $S\rightarrow \overline{\mathcal{J}}_{g,n}^{1,\times}$.

Let $V$ be a vertex operator algebra and let $M^\bullet=(M^1,\dots, M^n)$ be  $V$-modules. 
The action of $\textrm{Der}_+\, \mathcal{O}=\textrm{Lie}(\textrm{Aut}_+\mathcal{O})$ on each $M^i$ can be exponentiated to an action of $\textrm{Aut}_+\mathcal{O}$ on $M^i$. 

The action of $\mathfrak{L}(V)$ on each module $M^i$ via $\alpha_{M_i}\colon \mathfrak{L}(V)\rightarrow M^i$, compatible with the action of $\textrm{Aut}_+\mathcal{O}$ (see \S \ref{HCcompM}), 
induces an anti-homomorphism of Lie algebras
\begin{equation}
\label{alphaoMi}
\begin{split}
\begin{array}{r@{\;}ccl}
\alpha_{\otimes_{i=1}^n M^i}\colon & \mathfrak{L}(V)^{\oplus n} &\rightarrow & \textrm{End}(\otimes_{i=1}^n M^i)\\
& (u_1, \dots, u_n) & \mapsto & \sum_{i=1}^n \textrm{Id}_{M^1}\otimes \cdots \otimes \alpha_{M^i}(u_i)\otimes \cdots \otimes\textrm{Id}_{M^n}\\
\end{array}
\end{split}
\end{equation}
compatible with the action of $(\textrm{Aut}_+\mathcal{O})^n$.
This extends $\mathcal{O}_{\widetriangle{S}}\,$-linearly to an $(\textrm{Aut}_+\mathcal{O})^n$-equivar\-iant action of $\mathcal{L}(V)^n=\mathfrak{L}(V)^n \,\widehat{\otimes}\,\mathcal{O}_{\widetriangle{S}}$ on $\otimes_{i=1}^n M^i\otimes\,\mathcal{O}_{\widetriangle{S}}$.
That is, $\mathcal{L}(V)^n$ and $(\textrm{Aut}_+\mathcal{O})^n$ act compatibly on $\otimes_{i=1}^n M^i \otimes  \mathcal{O}_{\widetriangle{S}}$.
The sheaf $\mathcal{L} (V)^n$ is $(\textrm{Aut}_+\mathcal{O})^n$-equivariant.
By Lemma \ref{UVnoutAut}, the image of $\mathcal{L}_{\mathcal{C}\setminus\mathcal{P}_\bullet}(V)$ in $\mathcal{L} (V)^n$ is preserved by the action of $(\textrm{Aut}_+\mathcal{O})^n$. 
It follows that the quasi-coherent sheaf $\widetriangle{\mathbb{V}}_g(V; M^\bullet)|_{\widetriangle{S}}$ is an $(\textrm{Aut}_+\mathcal{O})^n$-equivariant $\mathcal{O}_{\widetriangle{S}}\,$-module, hence descends along the principal $(\textrm{Aut}_+\mathcal{O})^n$-bundle $\widetriangle{\pi}\colon\widetriangle{S}\rightarrow S$. 
Namely:

\begin{definition}\label{SheafOfCoinvariants}
The \textit{sheaf of coinvariants} $\mathbb{V}^J(V; M^\bullet)_{(\mathcal{C}\rightarrow S, P_\bullet, \tau_\bullet)}$ on $S$ is the quasi-coherent sheaf  
\[
\mathbb{V}^J(V; M^\bullet)_{(\mathcal{C}\rightarrow S, P_\bullet, \tau_\bullet)} := \left(\widetriangle{\pi}_* \, \widetriangle{\mathbb{V}}_g(V; M^\bullet)|_{\widetriangle{S}} \right)^{(\textrm{Aut}_+\mathcal{O})^n}
\]
whose sections are the $(\textrm{Aut}_+\mathcal{O})^n$-invariant sections of $\widetriangle{\mathbb{V}}_g(V; M^\bullet)|_{\widetriangle{S}}$. 
Similarly to \S \ref{SheafOfCoinvariantsTriangle}, one can remove the hypothesis that $\mathcal{C}\setminus P_\bullet(S)$ is affine over $S$, and the construction can be extended to obtain a sheaf on $\overline{\mathcal{J}}_{g,n}^{1,\times}$, denoted $\mathbb{V}^J_g(V; M^\bullet)$.
\end{definition}

The fiber of $\mathbb{V}^J_g(V;M^\bullet)$ over a point $(C,P_\bullet, \tau_\bullet)$ in $\overline{\mathcal{J}}_{g,n}^{1,\times}$ is canonically isomorphic to the space of coinvariants  \eqref{spacecon}. The sheaf $\mathbb{V}^J(V;M^\bullet)$ comes with a twisted logarithmic $\mathcal{D}$-module structure, detailed in \S\ref{projconnCJ}.
This construction is the result of the {localization functor} transforming modules for the Harish-Chandra pair $\left({\mathfrak{L}(V)^n}, (\textrm{Aut}_+\mathcal{O})^n \right)$ into twisted logarithmic $\mathcal{D}$-modules.

\subsubsection{The second descent to define ${\mathbb{V}}_g(V; M^\bullet)$ on $\overline{\mathcal{M}}_{g,n}$}
\label{SecondDescent}
The action of $(\mathbb{C}^\times)^n$ on $\otimes_i M^i$ obtained from its $\mathbb{Z}^n$-gradation as in \eqref{CactM} induces an action of $(\mathbb{C}^\times)^n$ on $\mathbb{V}^J_g(V; M^\bullet)$. Indeed, since $\mathbb{C}^\times$ is realized as the quotient $\text{Aut}\,\mathcal{O}/\text{Aut}_+\mathcal{O}$, it is enough to check that the action of the Lie algebra $\mathscr{L}_{C \setminus P}(V)$ on $\otimes_i M^i$ is  independent of the choice of coordinates at the fixed points. This follows from $\eqref{urhourho}$.

Let $\left(\mathcal{C}\rightarrow S, P_\bullet\right)$ be a family of stable pointed curves over a smooth base $S$ with $n$ sections $P_i\colon S\rightarrow \mathcal{C}$, and assume that $\mathcal{C}\setminus P_\bullet(S)$ is affine over $S$ (e.g., $S=\mathcal{M}_{g,n}$).
Define the principal $(\mathbb{C}^\times)^n$-bundle $S^J\rightarrow S$ as the pull-back of $j\colon \overline{\mathcal{J}}^\times_{g,n}\rightarrow\overline{\mathcal{M}}_{g,n}$ via the moduli map $S\rightarrow \overline{\mathcal{M}}_{g,n}$.
Assume that for each $i$, there exists $a_i$ such that $L_0(v)=\deg(v)+a_i$ for homogeneous $v\in M^i$.

\begin{definition}
\label{SOCM}
The \textit{sheaf of coinvariants} $\mathbb{V}(V; M^\bullet)_{(\mathcal{C}\rightarrow S, P_\bullet)}$ on $S$ is the quasi-coherent sheaf  
\[
\mathbb{V}(V; M^\bullet)_{(\mathcal{C}\rightarrow S,P_\bullet)} := \left({j}_* \, \mathbb{V}^J_g(V; M^\bullet)|_{S^J} \otimes \otimes_{i=1}^n\,\omega_{\mathcal{C}/S}^{\otimes -a_i} \right)^{(\mathbb{G}_m)^n} \otimes \otimes_{i=1}^n\,\omega_{\mathcal{C}/S}^{\otimes a_i}.
\]
This can be extended as in \S \ref{SheafOfCoinvariantsTriangle} to obtain a sheaf on $\overline{\mathcal{M}}_{g,n}$, denoted $\mathbb{V}_g(V; M^\bullet)$.
\\ The \textit{sheaf of conformal blocks} on $\overline{\mathcal{M}}_{g,n}$ is the dual of the sheaf of coinvariants $\mathbb{V}_g(V; M^\bullet)$.
\end{definition}

The fiber of $\mathbb{V}_g(V; M^\bullet)$ over a $\mathbb{C}$-point $(C,P_\bullet)$ in $\overline{\mathcal{M}}_{g,n}$ is canonically isomorphic to the space of coinvariants \eqref{spaceconCstar}.

\begin{remark}By Theorem \ref{thm:POV}, one has that {\em{propogation of vacua}} holds for
$\mathbb{V}_g(V; M^\bullet)$.
\end{remark}

\subsubsection{Prior work on extensions to stable curves}\label{priorExtensions}
Sheaves of coinvariants were originally constructed from integrable modules at a fixed level $\ell$ over affine Lie algebras $\hat{\mathfrak{g}}$ associated to semi-simple Lie algebras $\mathfrak{g}$ on $\widehat{\mathcal{M}}_{0,n}=\widetriangle{\mathcal{M}}_{0,n}\setminus \Delta$ in  \cite{TK}.   These were generalized in \cite{tuy} to $\widetriangle{\mathcal{M}}_{g,n}$, and shown to be locally free of finite rank, with a number of other good  properties including a projectively flat logarithmic connection. In \cite{ts} they were shown to be  coordinate free, defined on $\overline{\mathcal{M}}_{g,n}$, and as is done here, an action of an Atiyah algebra on the bundles was explicitly computed.  

Prior to our work, and aside from aspects of \cite{bd} and \cite{bzf},
considered here for stable curves with singularities,  there has been some interest in sheaves 
of conformal blocks using vertex algebras.  In \cite{u} sheaves defined by representations at level $\ell$ over the Heisenberg vertex algebra were shown to be isomorphic to the space of theta functions of order $\ell$ on the curve, identified under pullback along the Abel-Jacobi map, and shown to be of finite rank. Sheaves defined using  {\em{regular}} vertex  algebras defined on $\widetriangle{\mathcal{M}}_{0,n}$ (regular vertex  algebras are conformal and satisfy additional properties including, but not limited to, $C_2$-cofiniteness, which guarantees that they have finitely many simple modules),  were shown to be of finite rank, and carry a projectively flat logarithmic connection, as well as satisfy other good properties in parallel to the classical case \cite{NT}. The authors \cite{NT} remark that their definition of conformal blocks agrees with that described in  \cite[\S\S 0.9, 0.10, 4.7.2]{bd}, although their construction appears very different than what we consider here.   Conformal blocks defined over a single smooth pointed curve with coordinates using modules over $C_2$-cofinite vertex algebras were shown to be finite-dimensional in \cite{NT}.  
Recently, sheaves of coinvariants defined by holomorphic vertex algebras (admitting only one $V$-module: $V$ itself) on  $\overline{\mathcal{M}}_{g,1}$ have been applied to the study of the Schottky problem and the slope of the effective cone of moduli of curves in \cite{gc}.  

Connections between conformal blocks defined by lattice vertex algebras and theta functions are discussed in \cite[\S 5]{bzf}.


\section{The twisted logarithmic $\mathcal{D}$-module structure on ${\mathbb{V}^J}(V; M^\bullet)$}
\label{projconnCJ}
In this section we specify the twisted logarithmic $\mathcal{D}$-module structure of the  quasi-coherent sheaf of coinvariants $\mathbb{V}^J_g(V;M^\bullet)$  (Definition~\ref{SheafOfCoinvariants}).  In particular, $\mathbb{V}^J_g(V;M^\bullet)$ supports a projectively flat logarithmic connection.   
To do this, we determine the (logarithmic) Atiyah algebra acting on  sheaves of coinvariants.  

Let $\Lambda :=\det \textbf{R}\pi_* \,\omega_{\overline{\mathcal{C}}_{g,n}/\overline{\mathcal{M}}_{g,n}}$ be the determinant of cohomology of the Hodge bundle  (see \S \ref{Hodge}),  $\Delta$  the divisor in $\overline{\mathcal{M}}_{g,n}$ of singular curves, and $\mathcal{A}_\Lambda$ the corresponding logarithmic Atiyah algebra. One has:

\begin{theorem}
\label{AtalgonCJ}
The logarithmic Atiyah algebra $\frac{c}{2}\mathcal{A}_\Lambda$ acts on $\mathbb{V}^J_g(V; M^\bullet)$.
\end{theorem}

Here $c$ is the central charge of the vertex algebra $V$.
For $c=0$, this gives a logarithmic $\mathcal{D}$-module structure on $\mathbb{V}^J_g(V; M^\bullet)$ (\S \ref{zero}). 
We review the definition of a logarithmic Atiyah algebra in~\S \ref{LogAtiyahAlgebra}.
For coinvariants of affine Lie algebras, the above statement is in \cite{tuy}; see also \cite[\S 7.10]{bk}.

\subsection{Logarithmic Atiyah algebras}\label{LogAtiyahAlgebra}
Let $S$ be a smooth scheme over $\mathbb{C}$, and $\Delta\subset S$ a normal crossing divisor. Following \cite{besh} and \cite{ts}, an \textit{Atiyah algebra with logarithmic singularities along $\Delta$} (abbreviated \textit{logarithmic Atiyah algebra}, when $\Delta$ is clear from the context) is a Lie algebroid $\mathcal{A}$ (see \S \ref{LieAlgebroids}) together with  its {\em{fundamental extension sequence}}: 
\[
0 \rightarrow \mathcal{O}_S \xrightarrow{\iota_{\mathcal{A}}} \mathcal{A} \xrightarrow{p_{\mathcal{A}}} \mathcal{T}_{S}(-\log \Delta) \rightarrow 0.
\] 
One defines sums via the Baer sum 
\[
\mathcal{A}+\mathcal{B}:=\mathcal{A}\times_{\mathcal{T}_{S}(-\log\Delta)}\mathcal{B}/\{(\iota_{\mathcal{A}}(f), -\iota_{\mathcal{B}}(f)), \mbox{ for $f\in\mathcal{O}_S$}\}, 
\]
with fundamental sequence given by $\iota_{\mathcal{A}+\mathcal{B}}=(\iota_{\mathcal{A}},0)=(0,\iota_{\mathcal{B}})$, and $p_{\mathcal{A}+\mathcal{B}}=p_{\mathcal{A}}=p_{\mathcal{B}}$, and scalar multiplications $\alpha \mathcal{A}=(\mathcal{O}_S\oplus \mathcal{A})/(\alpha,1)\mathcal{O}_S$, with $i_{\alpha\mathcal{A}}=(\textrm{id},0)$ and $p_{\alpha\mathcal{A}}=(0,p_{\mathcal{A}})$, for $\alpha\in \mathbb{C}$. 
When $m\in\mathbb{Z}_{> 0}$, one has  $m\mathcal{A}=\mathcal{A} + \cdots + \mathcal{A}$ (with m summands).
When $\Delta=\emptyset$, the above recovers the classical case of Atiyah algebras. Restrictions of Atiyah algebras over the subsheaf  $\mathcal{T}_{S}(-\log\Delta) \hookrightarrow \mathcal{T}_{S}$ are logarithmic Atiyah algebras.

We will focus on logarithmic Atiyah algebras arising from line bundles: for a line bundle $\mathcal{L}$ over a smooth scheme $S$ equipped with a normal crossing divisor $\Delta$, we denote by $\mathcal{A}_\mathcal{L}$ the logarithmic Atiyah algebra of first order differential operators acting on $\mathcal{L}$ which preserve $\Delta$. 
Choosing a local trivialization $\mathcal{L} \cong \mathcal{O}_S$ on an open subset $U$ of $S$,  the elements of $\mathcal{A}_\mathcal{L}(U)$ are  $D + f$, where $D \in \mathcal{T}_S(-\log \Delta)(U)$ and $f \in \mathcal{O}_S(U)$. For $\alpha\in\mathbb{C}$, while the line bundle $\mathcal{L}^\alpha$ may not be well-defined, 
the logarithmic Atiyah algebra $\mathcal{A}_{\mathcal{L}^\alpha}$ is defined as $\mathcal{A}_{\mathcal{L}^\alpha}:=\alpha \mathcal{A}_\mathcal{L}$.

The action of a logarithmic Atiyah algebra $\mathcal{A}$ on a quasi-coherent sheaf $\mathbb{V}$ over $S$ is an action of sections $\mu\in \mathcal{A}$ as first-order differential operators $\Phi(\mu)$ on $\mathbb{V}$ such that: (i) the principal symbol of $\Phi(\mu)$ is $p_\mathcal{A}(\mu)\otimes \textrm{id}_\mathbb{V}$;
(ii) $\Phi(\iota_\mathcal{A}(1))=1$. In particular, for $\alpha\in\mathbb{C}$ and a line bundle $\mathcal{L}$ on $S$, an action of $\mathcal{A}_{\mathcal{L}^\alpha}$ is an action of $\mathcal{A}_\mathcal{L}$ with the difference that $\iota_{\mathcal{A}_\mathcal{L}}(1)$ acts as multiplication by~$\alpha$.

A flat connection with logarithmic singularities on a quasi-coherent sheaf $\mathbb{V}$ over $S$ is defined as a map of sheaves of Lie algebras  $\mathcal{T}_S(- \log \Delta) \to \text{End}(\mathbb{V})$. Following \cite{beka} we call projectively flat connection with logarithmic singularities a map $\mathcal{T}_S(- \log \Delta) \to \text{End}(\mathbb{V})/ ( \mathcal{O}_S \,\text{Id}_\mathbb{V})$. This means that the action of $\mathcal{T}_S(- \log \Delta)$ on $\mathbb{V}$ is well defined only up to constant multiplication. From the fundamental sequence defining an Atiyah algebra $\mathcal{A}$, it follows that the action of $\mathcal{A}$ on a quasi-coherent sheaf $\mathbb{V}$ induces a projectively flat connection with logarithmic singularities on $\mathbb{V}$.

\subsection{The Hodge line bundle}\label{Hodge}
For $n>0$, let $n\textrm{Vir}$ be the quotient of the direct sum $\textrm{Vir}^{\oplus n}$ of $n$ copies of the Virasoro algebra modulo the identification of
the $n$ central elements in $\textrm{Vir}^{\oplus n}$ of type $(0, \dots,0, K,0, \dots, 0)$.  We denote by $\overline{K}$ the equivalence class of these elements. 
Then $n\textrm{Vir}$ is a central extension
\[
0\rightarrow \mathfrak{gl}_1\cdot \overline{K} \rightarrow n\textrm{Vir} \rightarrow (\textrm{Der}\, \mathcal{K})^{\oplus n} \rightarrow 0
\]
with Lie bracket  given by
\[
\left[(L_{p_i})_{1\leq i\leq n}, (L_{q_i})_{1\leq i\leq n}\right] = \left((p_i-q_i) L_{p_i+q_i} \right)_{1\leq i\leq n} + \frac{\overline{K}}{12} \sum_{i=1}^n (p_i^3-p_i)\delta_{p_i+q_i,0}.
\]

Let $\left( \mathcal{C}\rightarrow S, P_\bullet, t_\bullet\right)$ be a family of stable pointed curves of genus $g$ over a smooth base $S$ with $n$ sections $P_i\colon S\rightarrow \mathcal{C}$ and formal coordinates $t_i$. These data give a moduli map $S\rightarrow \widetriangle{\mathcal{M}}_{g,n}$. Assume that $\mathcal{C}\setminus P_\bullet(S)$ is affine ever $S$.
There is a Lie algebra embedding $\textrm{Ker}\, a\hookrightarrow n\textrm{Vir} \,\widehat{\otimes}\, \mathcal{O}_{S}$ \cite[\S 19.6.5]{bzf}, where $a$ is the anchor map from \eqref{anchor}. 
The quotient 
\[
n\textrm{Vir}(\mathbb{C}) \,\widehat\otimes_\mathbb{C}\, \mathcal{O}_{S}\Big/\textrm{Ker}\, a
\]
is anti-isomorphic to an extension
\[
0 \rightarrow \mathcal{O}_{S} \rightarrow \mathcal{A} \rightarrow \mathcal{T}_{S}(-\log \Delta) \rightarrow 0
\]
and carries the structure of a logarithmic Lie algebroid (see \S \ref{LieAlgebroids}). 
The following statement is known as the \textit{Virasoro uniformization} for  
the line bundle $\Lambda$ (see Fig.~\ref{fig:Virunifla}).

\begin{theorem}[\cite{adkp}, \cite{besh}, \cite{kvir}, \cite{tuy}]
\label{nVirAt}
The logarithmic Atiyah algebra $\mathcal{A}$ is isomorphic to $\frac{1}{2}\mathcal{A}_\Lambda$.
\end{theorem}

\begin{proof}[Sketch of the proof]
The Lie algebra $n\textrm{Vir}$ acts on $\Lambda$ extending the action of $(\textrm{Der}\, \mathcal{K})^{\oplus n}$ on $\widetriangle{\mathcal{M}}_{g,n}$
from Theorem \ref{Virunif} so that the element $\overline{K}\in n\textrm{Vir}$ acts as multiplication by $2$ (the case $n=1$ is in \cite{kvir}).
This action extends to an anti-homomorphism of logarithmic Lie algebroids
\[
n\textrm{Vir}(\mathbb{C}) \,\widehat\otimes_\mathbb{C}\, \mathcal{O}_{S} \rightarrow \mathcal{A}_\Lambda
\]
such that $1\in \mathcal{O}_{S} \hookrightarrow \mathcal{A}$ is mapped to $\frac{1}{2}\in \mathcal{O}_{S} \hookrightarrow \mathcal{A}_\Lambda$.
The subsheaf $\textrm{Ker}\, a\hookrightarrow n\textrm{Vir}(\mathbb{C}) \,\widehat{\otimes}_\mathbb{C}\, \mathcal{O}_{S}$ --- whose fiber at $(C,P_\bullet,t_\bullet)$ consists of the Lie algebra $\mathscr{T}_C(C\setminus P_\bullet)$ 
of regular vector fields on $C\setminus P_\bullet$ --- acts trivially on $\Lambda$ (e.g., \cite[(7.10.11)]{bk} shows that $\textrm{Ker}\, a$ acts trivially on $\Lambda^{-1}$).
This gives a map of Atiyah algebras $\varphi\colon \mathcal{A}\rightarrow \frac{1}{2}\mathcal{A}_\Lambda$. Since  Atiyah algebras are in $\textrm{Ext}^1\left(\mathcal{T}_{S}(-\log \Delta), \mathcal{O}_{S}\right)$, the Five Lemma implies that $\varphi$ is an isomorphism.
\end{proof}

\begin{figure}[htb]
\centering
\begin{tikzcd}
\mathcal{O}_{S}\arrow[hook]{d} \arrow{r} & \mathcal{O}_{S} \arrow[hook]{d}\\
n\textrm{Vir}(\mathbb{C}) \,\widehat\otimes_\mathbb{C}\, \mathcal{O}_{S} \arrow[two heads]{r} \arrow[two heads]{d} & \mathcal{A}_\Lambda \arrow[two heads]{d}\\
\left(\textup{Der}\, \mathcal{K}(\mathbb{C})\right)^{\oplus n} \,\widehat\otimes_\mathbb{C}\, \mathcal{O}_{S} \arrow[two heads]{r}{a} & \mathcal{T}_{S} (-\log \Delta)
\end{tikzcd}
   \caption{The Virasoro uniformization for the line bundle $\Lambda$.}
   \label{fig:Virunifla}
\end{figure}

\subsection{The action on the sheaf of coinvariants}\label{ActionCov}

Before the proof of Theorem \ref{AtalgonCJ}, we discuss some auxiliary results.

The Lie algebra $\mathfrak{L}(V)^n$ contains $\textrm{Vir}^n$ as a Lie subalgebra. The adjoint representation of $\textrm{Vir}^n$ on $\mathfrak{L}(V)^n$ given by $v\mapsto [\cdot,v]$, for $v\in \textrm{Vir}^n$, factors through an action of $n\textrm{Vir}$ on $\mathfrak{L}(V)^n$. 
The Lie algebroid $n\textrm{Vir}(\mathbb{C})  \,\widehat\otimes_\mathbb{C}\, \mathcal{O}_{\widetriangle{\mathcal{M}}_{g,n}}$ acts on $\mathcal{L}(V)^n$ as
\begin{align}
\label{nVirOactiononUVn}
(v\otimes f)\ast(u\otimes h):= [u,v]\otimes fh + u\otimes \left( a(v\otimes f)\cdot h \right),
\end{align}
for $v\in n\textrm{Vir}$, $u\in \mathfrak{L}(V)^n$, and local sections $f,h$ of $\mathcal{O}_{\widetriangle{\mathcal{M}}_{g,n}}$. (The action is given by an \textit{anti-homomorphism} of sheaves of Lie algebras from $n\textrm{Vir}  \,\widehat\otimes\, \mathcal{O}_{\widetriangle{\mathcal{M}}_{g,n}}$ to the sheaf of endomorphisms of $\mathcal{L}(V)^n$. Indeed, the map $v\mapsto [\cdot,v]$, for $v\in n\textrm{Vir}$,  and the anchor map $a$ are both anti-homomorphisms. 
Here, the anchor map $a$ applies to the image of $v\otimes f$ in $(\textrm{Der}\, \mathcal{K}(\mathbb{C}))^n \,\widehat\otimes_\mathbb{C}\, \mathcal{O}_{\widetriangle{\mathcal{M}}_{g,n}}$.

\begin{lemma}
\label{UVnoutnVir}
The image of $\mathcal{L}_{\mathcal{C}\setminus\mathcal{P}_\bullet}(V)$ inside $\mathcal{L}(V)^n$ is preserved by the action of $n\textup{\textrm{Vir}}(\mathbb{C}) \,\widehat\otimes_\mathbb{C}\, \mathcal{O}_{\widetriangle{\mathcal{M}}_{g,n}}$ from~\eqref{nVirOactiononUVn}.
\end{lemma}

\begin{proof}
The case $n=1$ and zero central charge $c=0$ is \cite[Thm 17.3.11]{bzf}. 
The general case is similar; we sketch here a proof. 

Let $\widehat{(\textrm{Aut}\,\mathcal{K})^n}$ be the group ind-scheme satisfying $\textrm{Lie}\,\widehat{(\textrm{Aut}\,\mathcal{K})^n}=n\textrm{Vir}$.
The action of $(\textrm{Aut}\,\mathcal{K})^n$ on $\widetriangle{\mathcal{M}}_{g,n}$ induces an action of $\widehat{(\textrm{Aut}\,\mathcal{K})^n}$ on $\widetriangle{\mathcal{M}}_{g,n}$.
Moreover, by exponentiating the action of $n\textrm{Vir}$, one also has an action of $\widehat{(\textrm{Aut}\,\mathcal{K})^n}$ on $V^n$. 
(The exponential of the action of $n\textrm{Vir}$ on $V^n$ is well defined: this follows from the fact that the action of $L_p$ with $p>0$ on $V$ has negative degree $-p$, and can be integrated since $V$ has gradation bounded from below; the action of $L_0$ on $V$ has integral eigenvalues, and thus integrates to a multiplication by a unit; finally, the action of $L_p$ with $p<0$ integrates to an action of a nilpotent element; see \cite[\S 7.2.1]{bzf}).
This gives rise to an action of $\widehat{(\textrm{Aut}\,\mathcal{K})^n}$ on $\mathfrak{L}(V)^n$ by conjugation. 
Thus, $\widehat{(\textrm{Aut}\,\mathcal{K})^n}$ acts on $\mathcal{L}(V)^n$ via
\begin{align}
\label{hatAutKnonUVn}
(u\otimes h) \cdot \rho := (\rho^{-1} u \rho) \otimes (h\cdot \rho) ,
\end{align}
for $u\otimes h$ in $\mathcal{L}(V)^n$ and $\rho$ in $\widehat{(\textrm{Aut}\,\mathcal{K})^n}$.

The differential of this action is an action of $n\textrm{Vir}\,\widehat\otimes_\mathbb{C}\, H^0(\widetriangle{\mathcal{M}}_{g,n},\mathcal{O}_{\widetriangle{\mathcal{M}}_{g,n}})$ on $\mathcal{L}(V)^n$ given by
\begin{align}
\label{nViractiononUVn}
(v\otimes f)\ast(u\otimes h):= [u,v]\otimes fh + u\otimes \left( \alpha(v\otimes f)\cdot h \right),
\end{align}
for $u\otimes h \!\in\! \mathcal{L}(V)^n$ and $v\otimes f \!\in\! n\textrm{Vir}\,\widehat\otimes_\mathbb{C}\, H^0(\widetriangle{\mathcal{M}}_{g,n},\mathcal{O}_{\widetriangle{\mathcal{M}}_{g,n}})$. 
The action of $n\textrm{Vir}(\mathbb{C})  \,\widehat\otimes_\mathbb{C}\, \mathcal{O}_{\widetriangle{\mathcal{M}}_{g,n}}$ on $\mathcal{L}(V)^n$ from \eqref{nVirOactiononUVn} is the $\mathcal{O}_{\widetriangle{\mathcal{M}}_{g,n}}$-linear extension of the action \eqref{nViractiononUVn} of $n\textrm{Vir}$ on $\mathcal{L}(V)^n$.
(Recall from Theorem \ref{Virunif} that the anchor map $a$ is the $\mathcal{O}_{\widetriangle{\mathcal{M}}_{g,n}}$-linear extension of the map $\alpha$.)
Hence, to prove the statement, it is enough to show that the action \eqref{hatAutKnonUVn} of $\widehat{(\textrm{Aut}\,\mathcal{K})^n}$ on $\mathcal{L}(V)^n$ preserves the image of $\mathcal{L}_{\mathcal{C}\setminus\mathcal{P}_\bullet}(V)$ in $\mathcal{L}(V)^n$.

Using the action of $\widehat{(\textrm{Aut}\,\mathcal{K})^n}$ on $\widetriangle{\mathcal{M}}_{g,n}$, denote by $(C^\rho, P_\bullet^\rho, t_\bullet^\rho)$ the image of a coordinatized $n$-pointed curve $(C,P_\bullet, t_\bullet)$ via the action of $\rho$ in $\widehat{(\textrm{Aut}\,\mathcal{K})^n}$. Let $\mathscr{L}'_{C\setminus P}(V)$ and $\mathscr{L}'_{C^\rho\setminus P_\bullet^\rho}(V)$ be the fibers of the subsheaf
\[
\textrm{Im}\left(\mathcal{L}_{\mathcal{C}\setminus\mathcal{P}_\bullet}(V)\rightarrow \mathcal{L}(V)^n\right) \subset \mathcal{L}(V)^n
\]
on $(C,P_\bullet, t_\bullet)$ and $(C^\rho, P_\bullet^\rho, t_\bullet^\rho)$, respectively.
It is enough to show that
\begin{align}
\label{urhourho}
\mathscr{L}'_{C^\rho\setminus P_\bullet^\rho}(V) = \rho^{-1} \mathscr{L}'_{C\setminus P}(V) \rho
\end{align}
for $\rho \in \widehat{\textrm{Aut}}\,\mathcal{K}$. As in \cite[\S 17.3.12]{bzf}, this follows from a generalization of a result of Y.-Z.~Huang
\cite[Prop.~7.4.1]{yzhuang}, \cite[\S 17.3.13]{bzf}.
\end{proof}

Let
\begin{align*}
\alpha_{\otimes_{i=1}^n M^i}\colon \mathfrak{L}(V)^{\oplus n}\rightarrow \textrm{End}(\otimes_{i=1}^n M^i)
\end{align*}
be the anti-homomorphism of Lie algebras induced by the action of $\mathfrak{L}(V)$ on each module $M^i$ as in \eqref{alphaoMi}. This map restricts to an action of 
$\textrm{Vir}^n$, and since the central elements $(0, \dots, 0, K, 0, \dots, 0)$ all act as multiplication by the central charge $c$, this action factors through an action of $n\textrm{Vir}$ on $\otimes_{i=1}^n M^i$.
The Lie algebroid $n\textrm{Vir}(\mathbb{C})  \,\widehat\otimes_\mathbb{C}\, \mathcal{O}_{\widetriangle{\mathcal{M}}_{g,n}}$ acts on $\otimes_{i=1}^n M^i\otimes_\mathbb{C}\mathcal{O}_{\widetriangle{\mathcal{M}}_{g,n}}$ as
\begin{align}
\label{nVirOactiononMO}
(v\otimes f)\diamond(w\otimes h):= \left( \alpha_{\otimes_{i=1}^n M^i}(v)\cdot w \right)\otimes fh + w\otimes \left( a(v\otimes f)\cdot h \right)
\end{align}
for $v\in n\textrm{Vir}$, $w\in \otimes_{i=1}^n M^i$, and local sections $f,h$ of $\mathcal{O}_{\widetriangle{\mathcal{M}}_{g,n}}$.
Recall that the action of $\mathfrak{L}(V)^n$ on $\otimes_{i=1}^n M^i$ extends $\mathcal{O}_{\widetriangle{\mathcal{M}}_{g,n}}$-linearly to an action of
$\mathcal{L}(V)^n$  on $\otimes_{i=1}^n M^i\otimes_\mathbb{C}\mathcal{O}_{\widetriangle{\mathcal{M}}_{g,n}}$, and this induces an action of 
$\mathcal{L}_{\mathcal{C}\setminus\mathcal{P}_\bullet}(V)$ on $\otimes_{i=1}^n M^i\otimes_\mathbb{C}\mathcal{O}_{\widetriangle{\mathcal{M}}_{g,n}}$.

\begin{lemma}
\label{UVnoutMOnVir}
The subsheaf 
\begin{equation}
\label{UVnoutMO}
\mathcal{L}_{\mathcal{C}\setminus\mathcal{P}_\bullet}(V)\cdot \left(\otimes_{i=1}^n M^i\otimes_\mathbb{C}\mathcal{O}_{\widetriangle{\mathcal{M}}_{g,n}}\right)
\end{equation}
of $\otimes_{i=1}^n M^i\otimes_\mathbb{C}\mathcal{O}_{\widetriangle{\mathcal{M}}_{g,n}}$ is preserved by the action of $n\textup{\textrm{Vir}}(\mathbb{C}) \,\widehat\otimes_\mathbb{C}\, \mathcal{O}_{\widetriangle{\mathcal{M}}_{g,n}}$ from \eqref{nVirOactiononMO}.
\end{lemma}

\begin{proof}
We need to show that
\begin{equation}
\label{vfuhwt}
(v\otimes f)\diamond \left[(u\otimes h)\cdot (w\otimes t)\right] 
\end{equation}
is in \eqref{UVnoutMO},
for all local sections $v\otimes f \in n\textrm{Vir}(\mathbb{C})\,\widehat\otimes_\mathbb{C}\, \mathcal{O}_{\widetriangle{\mathcal{M}}_{g,n}}$,
$u\otimes h\in \mathcal{L}_{\mathcal{C}\setminus\mathcal{P}_\bullet}(V)$, and $w\otimes t \in \otimes_{i=1}^n M^i \otimes_\mathbb{C}\mathcal{O}_{\widetriangle{\mathcal{M}}_{g,n}}$.
In \eqref{vfuhwt} the first action is as in \eqref{nVirOactiononMO}, and the second one is 
\[
(u\otimes h)\cdot (w\otimes t) := \left( \alpha_{\otimes_{i=1}^n M^i}(u)\cdot w \right) \otimes ht.
\]
Expanding, one verifies that \eqref{vfuhwt} is equal to
\[
\left[ (v\otimes f)\ast (u\otimes h)\right] \cdot (w\otimes t) + (u\otimes h) \cdot \left((\alpha_{\otimes_{i=1}^n M^i}(v)\cdot w) \otimes ft\right).
\]
To show that this is in \eqref{UVnoutMO}, it is enough to show that $(v\otimes f)\ast (u\otimes h)$ is in the image of $\mathcal{L}_{\mathcal{C}\setminus\mathcal{P}_\bullet}(V)$ in $\mathcal{L}(V)^n$, where the action is as in \eqref{nVirOactiononUVn}. This follows from Lemma \ref{UVnoutnVir}.
\end{proof}

\begin{proof}[Proof of Theorem \ref{AtalgonCJ}]
Let $\left(\mathcal{C}\rightarrow S, P_\bullet, t_\bullet\right)$ be a family of stable pointed curves of genus $g$ over a smooth base $S$ with $n$ sections $P_i$ and formal coordinates $t_i$. These data give a moduli map $S\rightarrow \widetriangle{\mathcal{M}}_{g,n}$. 
From the construction of the sheaf of coinvariants, we can
reduce to the case when $\mathcal{C}\setminus P_\bullet(S)$ is affine over $S$ (e.g., $S=\widehat{\mathcal{M}}_{g,n}=\widetriangle{\mathcal{M}}_{g,n}\setminus \Delta$).

The Lie algebroid $n\textrm{Vir}(\mathbb{C})  \,\widehat\otimes_\mathbb{C}\, \mathcal{O}_{S}$ acts on $\otimes_{i=1}^n M^i\otimes_\mathbb{C}\mathcal{O}_{S}$ as in \eqref{nVirOactiononMO}.
By Lemma \ref{UVnoutMOnVir}, the subsheaf 
\[
\mathcal{L}_{\mathcal{C}\setminus\mathcal{P}_\bullet}(V)\cdot \left(\otimes_{i=1}^n M^i\otimes_\mathbb{C}\mathcal{O}_{S}\right)
\]
is preserved by this action.
Since the map $\alpha_{\otimes_{i=1}^n M^i}$ and the anchor map $a$ are compatible with the action of $(\textrm{Aut}_+\mathcal{O})^n$ (as in \S\ref{HCcompM} and \S \ref{HCcomp}), the action 
of $n\textrm{Vir}(\mathbb{C})  \,\widehat\otimes_\mathbb{C}\, \mathcal{O}_{S}$ on $\otimes_{i=1}^n M^i\otimes_\mathbb{C}\mathcal{O}_{S}$ from \eqref{nVirOactiononMO} is $(\textrm{Aut}_+\mathcal{O})^n$-equivariant.

By Lemma \ref{UCPVircinUCPVLemma}, the action of $\mathcal{L}_{\mathcal{C}\setminus\mathcal{P}_\bullet}(V)$ extends an action of $\textrm{Ker}\, a\subset n\textrm{Vir}(\mathbb{C}) \,\widehat\otimes_\mathbb{C}\, \mathcal{O}_{S}$.
It follows that \eqref{nVirOactiononMO} induces an $(\textrm{Aut}_+\mathcal{O})^n$-equivariant action of the Lie algebroid 
\[
n\textrm{Vir}(\mathbb{C}) \,\widehat\otimes_\mathbb{C}\, \mathcal{O}_{S}\Big/\textrm{Ker}\, a,
\]
hence an action of $\mathcal{A}$ on the sheaf $\widetriangle{\mathbb{V}}(M^\bullet)_{(\mathcal{C}\rightarrow S, P_\bullet, t_\bullet)}$.
Equivalently, for some $a\in \mathbb{C}$, there exists an $(\textrm{Aut}_+\mathcal{O})^n$-equivariant action of 
 the logarithmic Atiyah algebra $a\mathcal{A}$ on  $\widetriangle{\mathbb{V}}(M^\bullet)_{(\mathcal{C}\rightarrow S, P_\bullet, t_\bullet)}$,  which descends to an action of $a\mathcal{A}$ on the sheaf ${\mathbb{V}^J}(M^\bullet)_{(\mathcal{C}\rightarrow S, P_\bullet, \tau_\bullet)}$. 

To determine $a$, note that the central element $\overline{K} \in n\textrm{Vir}$ is the image of $(K,0,\dots,0)\in \textrm{Vir}^n\hookrightarrow \mathfrak{L}(V)^n$, which acts via \eqref{alphaoMi} as $c\cdot \textrm{id}$, by definition of the central charge $c$. It follows that $\overline{K}$ acts as $c\cdot \textrm{id}$, hence $a=c$. That is, the logarithmic Atiyah algebra $c\mathcal{A}$ acts on ${\mathbb{V}^J}(M^\bullet)_{(\mathcal{C}\rightarrow S, P_\bullet, \tau_\bullet)}$.
The statement follows from Theorem~\ref{nVirAt}.
\end{proof}

\subsection{The case $c=0$}\label{zero}
In the case of zero central charge $c=0$, 
the action in \eqref{nVirOactiononMO}  induces an $(\textrm{Aut}_+\mathcal{O})^n$-equivariant action of
\[
\left(\textrm{Der}\, \mathcal{K}(\mathbb{C})\right)^n \,\widehat\otimes_\mathbb{C}\, \mathcal{O}_{S}\Big/\textrm{Ker}\, a 
\cong \mathcal{T}_{S}(-\log \Delta)  
\]
on $\widetriangle{\mathbb{V}}(M^\bullet)_{(\mathcal{C}\rightarrow S, P_\bullet, t_\bullet)}$. This gives a logarithmic $\mathcal{D}$-module structure on $\mathbb{V}^J(M^\bullet)$ when $c=0$.


\section{The twisted logarithmic $\mathcal{D}$-module structure on ${\mathbb{V}}(V; M^\bullet)$} 
\label{ProjConnVM}

In Theorem \ref{At2} we specify the twisted logarithmic $\mathcal{D}$-module structure on the  quasi-coherent sheaf of coinvariants on $\overline{\mathcal{M}}_{g,n}$ under the assumption that the modules $M^i$ are simple. In particular $\mathbb{V}_g(V;M^\bullet)$ carries a projectively flat logarithmic connection.
The twisted $\mathcal{D}$-module structure on the  sheaf of coinvariants on the moduli space of \textit{smooth} pointed curves was first introduced in
\cite[\S 17.3.20]{bzf}. Here we extend the description over families of \textit{stable} pointed curves, and identify the twisted logarithmic $\mathcal{D}$-module by determining the precise logarithmic Atiyah algebra acting on sheaves of coinvariants. For sheaves of coinvariants of integrable representations of an affine Lie algebra, this statement was proved by Tsuchimoto  \cite{ts}.

To state the result we need some notation. Let $\Lambda:=\det \textbf{R}\pi_* \,\omega_{\overline{\mathcal{C}}_{g,n}/\overline{\mathcal{M}}_{g,n}}$ be the determinant of cohomology of the Hodge bundle on~$\overline{\mathcal{M}}_{g,n}$, and $\Psi_i=P_i^*(\omega_{\overline{\mathcal{C}}_{g,n}/\overline{\mathcal{M}}_{g,n}})$ the cotangent line bundle on $\overline{\mathcal{M}}_{g,n}$ corresponding to the $i$-th marked point. Let $\mathcal{A}_\Lambda$, $\mathcal{A}_{\Psi_i}$ be the corresponding logarithmic Atiyah algebras with respect to the divisor $\Delta$ in $\overline{\mathcal{M}}_{g,n}$ of singular curves (see \S \ref{LogAtiyahAlgebra}).  We show:

\begin{theorem}
\label{At2}
Let $M^i$ be simple $V$-modules of conformal dimension $a_i$. Then 
the logarithmic Atiyah algebra 
\begin{align}
\label{ati}
\frac{c}{2}\mathcal{A}_\Lambda +\sum_{i=1}^n a_i \mathcal{A}_{\Psi_i}
\end{align}
acts on the sheaf of coinvariants ${\mathbb{V}}_g(V;M^\bullet)$. In particular, ${\mathbb{V}}_g(V; M^\bullet)$ carries a projectively flat logarithmic connection. 
\end{theorem}

Here $V$ is a vertex operator algebra with central charge $c$ (see \S \ref{V}). Recall that a simple $V$-module $M$ has conformal dimension $a \in \mathbb{C}$ if $L_0v=(a+\deg v)v$, for homogeneous $v\in M$.

\proof 
From Theorem \ref{AtalgonCJ} we know that $\frac{c}{2}\mathcal{A}_{\Lambda}$ acts on ${\mathbb{V}}^J(M^\bullet)$. This means that for any family of stable pointed curves $(\mathcal{C}\rightarrow S, P_\bullet, \tau_\bullet)$ over a smooth base $S$ with $n$ sections $P_i\colon S\rightarrow \mathcal{C}$ and $n$ non-zero $1$-jets~$\tau_i$, the logarithmic Atiyah algebra $\frac{c}{2}\mathcal{A}_{\Lambda}$ acts on $\mathbb{V}^J(M^\bullet)_{(\mathcal{C} \to S, P_\bullet, \tau_\bullet)}$. We show that this naturally induces an action of $\frac{c}{2}\mathcal{A}_\Lambda +\sum_{i=1}^n a_i \mathcal{A}_{\Psi_i}$ on ${\mathbb{V}}(M^\bullet)_{(\mathcal{C} \to S, P_\bullet)}$. Assume for simplicity  $n=1$.

Since the sheaf  $\mathbb{V}^J(M)_{(\mathcal{C} \to S, P, \tau)}$ has an action of $\frac{c}{2}\mathcal{A}_\Lambda$,  the sheaf
$\mathbb{V}^J(M)_{(\mathcal{C} \to S, P, \tau)} \otimes \Lambda^{-\frac{c}{2}}$ has an action of $\mathscr{T}_{\overline{\mathcal{J}}^{\times}_{g,1}}(-\log \Delta)$. We refer to \cite[\S 2.4]{ts} for the definition of the formal line bundle $\Lambda^{-\frac{c}{2}}$ for arbitrary $c\in \mathbb{C}$.
This implies that ${\mathbb{V}}(M)_{(\mathcal{C} \to S, P)}\otimes \Lambda^{-\frac{c}{2}}$ is equipped with an action of the sheaf of $\mathbb{C}^\times$-invariant tangent vectors of $\overline{\mathcal{J}}^{\times}_{g,1}$ preserving the divisor of singular curves $\Delta$. Following \cite[Thm 5]{AtiyahConnection},  this sheaf coincides with $\mathcal{A}_{\Psi^{-1}}=-\mathcal{A}_{\Psi}$.
It follows that $\mathbb{V}(M)_{(\mathcal{C}\rightarrow S, P)} \otimes \Lambda^{-\frac{c}{2}}$ has an action of the Lie algebroid $ -\mathcal{A}_{\Psi}$. 

In order to conclude, it is enough to prove that the image of $1$ under the canonical map $\mathcal{O}_{S} \to -\mathcal{A}_{\Psi}$ acts on $\mathbb{V}(M)_{(\mathcal{C}\rightarrow S, P)}\otimes \Lambda^{-\frac{c}{2}}$ as multiplication by $-a$, where $a$ is the conformal dimension of the simple $V$-module $M$.
This can be verified locally on $S$. Fix a point $(C,P,0)$ on the divisor $Z:=\Psi \setminus \overline{\mathcal{J}}^{\times}_{g,1}$ and consider the local restriction of $\overline{\mathcal{J}}^{\times}_{g,1}$ near this point as $\textrm{Spec}(R[w,w^{-1}])$, where $R$ is the local ring of $C$ at $P$ and  the divisor $Z$ is locally given by $w=0$. Since $\mathbb{C}^\times$ acts on $w$ by multiplication, the ($\mathbb{C}^\times$)-equivariant vector fields are given by $\mathscr{T}_R \oplus w R \partial_w$. Locally the exact sequence defining $-\mathcal{A}_{\Psi}$ is then given by
\[ 
0 \to R \to  w R \partial_w \oplus \mathscr{T}_R  \to \mathscr{T}_R \to 0
\]
where $1\in R$ is mapped to $w\partial_w=-L_0$. Following the description of $\mathscr{M}_\mathcal{C}$ in \S\ref{descrMc}, we similarly identify the elements of $\mathbb{V}(M)_{(\mathcal{C}\rightarrow S, P)}$ as being represented by linear combinations of elements of the form $v \otimes w^{\deg v}$, where $v$ is a homogeneous element of $M$. Observe that since the action of $-\mathcal{A}_{\psi}$ is compatible with the action of $\mathcal{L}_{\mathcal{C}\setminus P}(V)$ on  $M\otimes \mathcal{O}_S$, we can carry the computation on $\mathbb{C}^\times$-invariant sections of $M\otimes \mathcal{O}_S$. The action of $w\partial w$ on $\mathbb{V}(M)_{(\mathcal{C}\rightarrow S, P)}$ is described in $\eqref{nVirOactiononMO}$. The computation
\begin{align*}
    w\partial w\left(v \otimes w^{\deg v}\right)  &= -L_0(v) \otimes w^{\deg v} + v \otimes w \partial_w w^{\deg v} \\
    &=-(a+\deg v)v \otimes w^{\deg v}  + v \otimes (\deg v) w^{\deg v}\\ 
    &=-a\left(v \otimes w^{\deg v}\right)
\end{align*}
concludes the argument.
\endproof

\subsection{Remark}
\label{Belkale}
While in this work we study the projectively flat connection $\nabla$ for sheaves on $\overline{\mathcal{M}}_{g,n}$ defined using vertex operator algebras along the lines of \cite{tuy, ts, beka, bk}, there are other approaches and further questions studied in the classical case, where conformal blocks are known to be  isomorphic to generalized theta functions (see e.g.,~\cite{Hi, FaltingsCon,  FelderCon, LaszloCon}).
For instance, the {\em{geometric unitary conjecture}} \cite{GLectures, FGScalar, GKChernSimons}  was an explicit proposal of a unitary metric on the Verlinde bundles that would be (projectively) preserved by the connection $\nabla$ on $\mathcal{M}_{g,n}$. The geometric unitary conjecture for $\mathfrak{s}\mathfrak{l}_{2}$ was proved by Ramadas \cite{Ramadas}, and for arbitrary Lie algebras in genus zero by Belkale \cite{BelkaleConnect}.   It would be interesting to know if this picture could be extended to bundles defined by modules over  vertex operator algebras.


\section{Chern classes of vector bundles $\mathbb{V}_g(V;M^\bullet)$ on $\mathcal{M}_{g,n}$}\label{Chern}

An expression for Chern classes of vector bundles on ${\mathcal{M}}_{g,n}$ of coinvariants of affine Lie algebras was given  in \cite{mop}.  
One of the key ingredients used in \cite{mop} is the explicit description of the projectively flat logarithmic connection from \cite{ts}.
In a similar fashion, the projectively flat logarithmic connection from Theorem \ref{At2}  allows to compute Chern classes of vector bundles of coinvariants of vertex algebras on ${\mathcal{M}}_{g,n}$ to arrive at  Corollary \ref{chMgn2} below.

Let $V$ be a vertex operator algebra of central charge $c$, and $M^1,\dots,M^n$ be simple  $V$-modules. 
Let $a_i$ be the conformal dimension of $M^i$, for each $i$, that is, $L_0(v)= \left( a_i + \deg v \right) v$, for homogeneous $v\in M^i$.
In this section, we assume that the sheaf of coinvariants $\mathbb{V}_g(V;M^\bullet)$ is  
the sheaf of sections of a vector bundle of finite rank on ${\mathcal{M}}_{g,n}$. 
This assumption is known to be true in special cases, as described in \S \ref{Remark}. 

We also assume that the central charge $c$ and the conformal dimensions $a_i$ are rational. For instance, after \cite{dlm}, this is satisfied for
a rational and $C_2$-cofinite vertex operator algebra $V$.
 
Let $\lambda:=c_1(\Lambda)$, and $\psi_i:=c_1(\Psi_i)$.
From Theorem \ref{At2} with $S=\mathcal{M}_{g,n}$, the action of the Atiyah algebra \eqref{ati}  gives a projectively flat connection on the sheaf of coinvariants $\mathbb{V}_g(V;M^\bullet)$ on $\mathcal{M}_{g,n}$.
This determines the Chern character of the restriction of $\mathbb{V}_g(V;M^\bullet)$ on $\mathcal{M}_{g,n}$. 

\begin{corollary}
\label{chMgn2}
When $\mathbb{V}_g(V;M^\bullet)$ has finite rank on $\mathcal{M}_{g,n}$ and $c, a_i\in \mathbb{Q}$, one has 
\[
{\rm ch}\left(  \mathbb{V}_g(V;M^\bullet)  \right) = {\rm rank} \,\mathbb{V}_g(V;M^\bullet) \, \cdot \, \exp\left( \frac{c}{2}\,\lambda + \sum_{i=1}^n a_i \psi_i \right)\, \in H^*(\mathcal{M}_{g,n}, \mathbb{Q}).
\]
Equivalently, the total Chern class is
\[
c\left(  \mathbb{V}_g(V;M^\bullet)  \right) =
\left( 1+ \frac{c}{2}\,\lambda + \sum_{i=1}^n a_i \psi_i \right)^{{\rm rank} \,\mathbb{V}_g(V;M^\bullet)} \in H^*(\mathcal{M}_{g,n}, \mathbb{Q}).
\]
\end{corollary}

\begin{proof}
The statement follows from the general fact that a $\mathcal{A}_{L^{\otimes a}}$-module $E$ of finite rank, with $L$ a line bundle and $a\in \mathbb{Q}$, satisfies $c_1(E)=({\rm rank}\, E) \cdot a \,c_1(L)$ (e.g., \cite[Lemma 5]{mop}); 
moreover, the projectively flat connection implies that ${\rm ch}(E)=({\rm rank}\, E) \cdot {\rm exp}(c_1(E)/{\rm rank}\, E)$  (e.g., \cite[(2.3.3)]{MR909698}).
\end{proof}

\subsection{Remark}
\label{Remark}
From Theorem \ref{At2},  the sheaf $\mathbb{V}_g(V;M^{\bullet})$ on $\mathcal{M}_{g,n}$ is equipped with a projectively flat connection. 
It follows that when $\mathbb{V}_g(V;M^{\bullet})$ is a sheaf of finite rank on $\mathcal{M}_{g,n}$, $\mathbb{V}_g(V;M^{\bullet})$  is also locally free \cite[\S 2.7]{sorger}.

It is natural to expect finite-dimensionality of spaces  of coinvariants constructed from vertex algebras with finitely many simple modules, when there is an expectation that the factorization property will hold \cite{bzf}.  This has been checked in special cases: For integrable highest weight representations at level $\ell$ of affine Lie algebras  \cite{tuy}; and highest weight representations of the so-called minimal series for the Virasoro algebra \cite{bfm}.  In the cases mentioned above, the vertex algebras are both rational and $C_2$-cofinite. Moreover, spaces of coinvariants associated to  modules over $C_2$-cofinite vertex  algebras have also been shown to be  finite-dimensional in other more general contexts \cite{an1, NT}.   While $C_2$-cofiniteness was  conjectured to be equivalent to rationality \cite{AbeBuhlDong}, this was disproved in \cite{AM}.  While slightly technical to define, $C_2$-cofiniteness  is natural \cite{ArakawaC2}, and satisfied by many types of vertex algebras  including a large class of simple $\mathcal{W}$-algebras containing all exceptional $\mathcal{W}$-algebras, and in particular the minimal series of principal $\mathcal{W}$-algebras discovered by Frenkel, Kac, and Wakimoto \cite{EFKW}, proved to be  $C_2$-cofinite  by Arakawa \cite{ArakawaC2W} (see also \cite{ArakawaRatW}), vertex algebras  associated to  positive-definite even lattices  \cite{yg, AbeLatticeRationality, jyLattice}, and a number of vertex algebras formed using the orbifold  construction \cite{AbeOrbifold, AbeOrbifold2, MiyamotoOrbifold, AiOrbifold}.

\subsection{Addendum} 
\label{sec:add}
In \cite{DGT2} we introduce a modification of the Lie algebra $\mathscr{L}_{C\setminus P_\bullet}(V)$ for which the results of the current paper still hold. The modified  Lie algebra allows us to show that for a vertex operator algebra $V$ which is rational, $C_2$-cofinite, and such that $V_0\cong \mathbb{C}$, coinvariants from finitely generated admissible $V$-modules 
satisfy the factorization and sewing properties, thus giving rise  to vector bundles of coinvariants on $\overline{\mathcal{M}}_{g,n}$. 
Consequently, we  show in \cite{DGT3} that if $V$ is also simple and self-contragredient, their Chern characters give rise to  cohomological field theories.
In particular,  Chern classes over $\overline{\mathcal{M}}_{g,n}$ lie in the tautological ring.

\bibliographystyle{abbrv}
\bibliography{Biblio}

\end{document}